\documentclass[12pt,centertags,oneside]{amsart}
\usepackage{amsmath,amstext,amsthm,amscd,typearea,hyperref}
\usepackage{amssymb}
\usepackage{a4wide}
\usepackage[mathscr]{eucal}
\usepackage{mathrsfs}
\usepackage{typearea}
\usepackage{charter}
\usepackage{pdfsync}
\usepackage[a4paper,width=16.5cm,top=3cm,bottom=2.8cm]{geometry}

\numberwithin{equation}{section}
\usepackage{cleveref}  

\allowdisplaybreaks
\emergencystretch=\maxdimen
\usepackage{multicol}

\usepackage{xcolor}

\newtheorem{theorem}{Theorem}[section]
\newtheorem{definition}[theorem]{Definition}
\newtheorem{proposition}[theorem]{Proposition}
\newtheorem{corollary}[theorem]{Corollary}
\newtheorem{lemma}[theorem]{Lemma}

\newtheorem{example}[theorem]{Example}

\newtheorem{conjecture}[theorem]{Conjecture}

\theoremstyle{definition}
\newtheorem{remark}[theorem]{Remark}

\newcommand{\mf}{m_{_{\mathrm{faithful}}}}
\newcommand{\g}{\mathfrak{g}}
\newcommand{\ZZ}{\mathrm{Z}}
\newcommand{\Z}{\mathbb{Z}}
\newcommand{\C}{\mathbb{C}}
\newcommand{\Q}{\mathbb{Q}}
\newcommand{\Hom}{\mathrm{Hom}}

\newcommand{\F}{\mathbb{F}}

\newcommand{\GL}{\mathrm{GL}}

\newcommand{\ad}{\mathrm{ad}}
\newcommand{\Iex}{I_\mathrm{ex}}

\newcommand{\rank}{\mathrm{rk}}
\newcommand{\f}{\mathfrak{f}}
\newcommand{\FZ}{\mathbb{Z}/p\mathbb{Z}}
\newcommand{\1}{\mathbf{1}}
\newcommand{\bb}{{\bf b}}
\newcommand{\R}{\mathrm{R}}

\newcommand{\lev}{\mathbf{lev}}

\newcommand{\G}{\mathrm{G}}

\newcommand{\m}{\mathfrak{m}}
\newcommand{\OO}[1]{\mathcal{O}/\mathfrak{p}^{#1}}
\newcommand{\restr}[2]{#1_{|_{#2}}}
\newcommand{\GG}{\mathscr{G}}

\newcommand{\T}{\mathbf{T}}

\newcommand{\vv}{{\bf v}}
\newcommand{\ww}{{\bf w}}
\newcommand{\zz}{{\bf z}}

\newcommand{\E}{\mathscr{E}}

\DeclareMathOperator{\Ann}{Ann}

\renewcommand{\i}{\mathbf{i}}
\renewcommand{\j}{\mathbf{j}}
\renewcommand{\k}{\mathbf{k}}

\newcommand{\fix}{\mathrm{Fix}}

\newcommand{\sv}{\mathsf v}

\newcommand{\D}{\mathscr{D}}
\newcommand{\A }{\mathscr{A}}

\renewcommand{\contentsname}{Contents\\{\footnotesize\normalfont(A table
of contents should normally not be included)}}

\newcommand{\Hall}{\mathscr{H}}

\usepackage[normalem]{ulem}

\usepackage{fancyhdr}

\title{Polynomiality of the faithful dimension for 
nilpotent groups\\ over finite truncated valuation rings }

\author{Mohammad Bardestani}
\address{John Abbott College, 21 275 Rue Lakeshore Road, Sainte-Anne-de-Bellevue, QC, H9X 3L9, Canada.}
\email{mohammad.bardestani@gmail.com}

\author{Keivan Mallahi-Karai}
\address{Constructor University, Campus Ring I, Bremen 28759 Germany}
\email{kmallahikarai@constructor.university}

\author{Dzmitry Rumiantsau}
\address{Constructor University, Campus Ring I, Bremen 28759 Germany}
 \email{drumiantsau@constructor.university}

\author{Hadi Salmasian}
\address{Department of Mathematics, University of Ottawa, STEM Complex,
150 Louis-Pasteur Pvt,
Ottawa, ON,
Canada K1N 6N5}
\email{hadi.salmasian@uottawa.ca}

\begin{document}

\begin{abstract}
Given a finite group $\G$, the {\it faithful dimension} of $\G$ over $\C$, denoted by $m_\mathrm{faithful}(\G)$, is the smallest integer $n$ such that $\G$ can be embedded in $\GL_n(\C)$. Continuing the work initiated in~\cite{BMS19}, 
we address the problem of determining 
the faithful dimension of a finite $p$-group of the form $\GG_R:=\exp(\g_R)$  
associated to 
$\g_R:=\g \otimes_\Z R $ in the Lazard correspondence, where 
$\g$ is  a nilpotent $\Z$-Lie algebra and $R$ ranges over finite truncated valuation rings.

Our first main result is that if $R$ is a finite field with $p^f$ elements and $p$ is sufficiently large, then
$m_\mathrm{faithful}(\GG_R)=fg(p^f)$ where $g(T)$ belongs to a finite list of polynomials $g_1,\ldots,g_k$,
with non-negative integer coefficients.
The latter list of polynomials is 
uniquely determined by the Lie algebra $\g$. Furthermore, for each $1\le i\leq k$ the set of pairs $(p,f)$ for which $g=g_i$ is a finite union of Cartesian products $\mathscr P\times \mathscr F$, where $\mathscr P$ is a Frobenius set of prime numbers and $\mathscr F$ is a subset of $\mathbb N$ that belongs to the Boolean algebra generated by arithmetic progressions.
Previously, existence of such a polynomial-type formula for   
$m_\mathrm{faithful}(\GG_R)$ was only established under the assumption that either $f=1$ or  $p$ is fixed.

Next we formulate a conjectural polynomiality property
for the value of $m_\mathrm{faithful}(\GG_R)$ in the more general setting where $R$ is a finite truncated valuation ring, and prove special cases of this conjecture. In particular, we show that for 
a vast class of Lie algebras $\g $ that are defined by partial orders, 
$m_\mathrm{faithful}(\GG_R)$
 is given  by a single polynomial-type formula.

Finally, we compute $m_\mathrm{faithful}(\GG_R)$ precisely in the case where 
$\g$ is the free metabelian nilpotent Lie algebra of class $c$ on $n$ generators and 
 $R$ is a finite truncated valuation ring.

\end{abstract}

\clearpage\maketitle
\thispagestyle{empty}


\medskip


\setcounter{tocdepth}{1}


\section{Introduction}
\label{sec:intro}
The \emph{faithful dimension} of a finite group $\G$ over $\C$  is
the smallest possible dimension of a faithful complex representation of $\G$. 
Throughout this paper we denote the faithful dimension of $\G$ by 
$m_{\mathrm{faithful}}(\G)$.
It is known
that $\mathrm{ed}_\C (\G)\leq m_{\mathrm{faithful}}(\G)$ for a finite group $\G$ (see~\cite[Proposition 4.15]{BerFav}), where $\mathrm{ed}_\mathbb K (\G)$ 
denotes the essential dimension of $\G$ over a field $\mathbb K$ as defined  by Buhler and 
Reichstein \cite{Reichstein}. It follows from a theorem of Karpenko and Merkurjev~\cite{Merkurjev}
that  equality holds for all finite $p$-groups (see also~\cite{MerkurjevI}). 

The current work is a continuation of~\cite{BMS19} in which the question of determining the faithful dimension was systematically studied for finite $p$-groups associated to $\Z$-Lie algebras via the 
\emph{Lazard correspondence}.  Let $\mathfrak f$ be a  
finite $\Z$-Lie algebra which is nilpotent and has cardinality $p^k$ for a prime number $p$ that is strictly larger than the nilpotency class of $\mathfrak f$. Then to $\mathfrak f$ we can associate a finite $p$-group with  underlying set $\mathfrak f$  and with
group multiplication defined by the Baker-Campbell-Hausdorff formula, that is
 \begin{align*}
x*y:=
\sum_{n>0}\frac{(-1)^{n+1}}{n}\sum_{\substack{(a_1,b_1),\dots, (a_n,b_n)\\ a_j+b_j\geq 1}}\frac{(\sum_{1\leq i\leq n} a_i+b_i)^{-1}}{a_1!b_1!\cdots a_n!b_n!}(\ad_x)^{a_1}(\ad_y)^{b_1}\dots (\ad_x)^{a_n}(\ad_y)^{b_n-1}(y),\notag
\end{align*}
for $x,y\in \mathfrak f$. By a result of Lazard~\cite[Chapter 9]{Khukhro} every finite $p$-group of nilpotency class strictly smaller than $p$ is obtained in this way from a unique Lie algebra $\mathfrak f$. It is natural to denote this $p$-group by $\exp(\mathfrak f)$. 
 
A general class of examples of the above construction can be given as follows. Let $\g$ be a nilpotent $\Z$-Lie algebra that is finitely generated as an abelian group, and let $R$ be a finite commutative unital ring of cardinality  $p^k$ where $p$ is a prime number strictly larger than the nilpotency class of $\g$. Then the above construction yields a finite $p$-group $\GG_R:=\exp(\g_R)$ where $\g_R:=\g\otimes_\Z R$.  
For example, for $p\geq n\geq 2$ the group of $n\times n$ unitriangular matrices over $R$  is obtained in this fashion from the Lie algebra of $n\times n$ strictly upper triangular matrices over $R$ (which has nilpotency class $n-1$).


Let $\F_q$ denote a finite field of order $q:=p^f$. A central  result of our previous work~\cite{BMS19} is about the dependence of   $m_{\mathrm{faithful}}(\GG_{\F_q})$ on $p$  and $f$. More precisely, we proved the following two assertions (see Theorems 2.5 and 2.7 of~\cite{BMS19}): 
\begin{itemize}

\item[(A)] There exist finitely many polynomials $g_1,\ldots,g_k$ with non-negative integer coefficients such that for every prime $p$ we have $m_\mathrm{faithful}(\GG_{\F_p})=g(p)$ where $g=g_i$ for some $1\leq i\leq k$. The set of $p$ for which $g=g_i$ is  definable in terms of  existence of solutions of polynomial systems modulo $p$.

\item[(B)] Given any sufficiently large prime $p$, there exist finitely many polynomials $g_1,\ldots,g_k$, depending only on $\g $ and on $p$, such that for every $q:=p^f$ with $f\in\mathbb N$ we have $m_\mathrm{faithful}(\GG_{\F_{q}})=fg(q)$ where $g=g_i$ for some $1\leq i\leq k$. The set of $f\in\mathbb N$ for which $g=g_i$ is a finite union of arithmetic progressions. 
\end{itemize}  

In statement A only $p$ varies, while in statement B only $f$ varies and it is not clear how the $g_i$ depend on $p$. Statements A and B  suggest that 
$m_\mathrm{faithful}(\GG_{\F_q})$ should possess 
a stronger polynomiality-type property,  where both $p$ and $f$ are allowed to vary simultaneously. We did not succeed in proving  a unifying assertion in \cite{BMS19} (see Remark 2.8 of the latter reference). 

The first major result of the present paper 
(Theorem~\ref{thm:mixed}) proves such a unifying assertion. 
For a polynomial $g(T) \in\mathbb Z[T]$, we denote the set of primes $p$ for which the congruence
$g(T) \equiv 0\ (\mathrm{mod}\ p)$ has a solution
by $V_g$. We call a set of prime numbers a \emph{Frobenius set} if it is in the Boolean algebra generated by the sets $V_g$. The Frobenius sets are closely related to Serre's Frobenian sets (see \cite{BMS19} or \cite{Lagarias} for a precise statement). Henceforth, by an arithmetic progression in $\mathbb N$ we mean a subset of $\mathbb N$ of the form $\left\{a+kb\,:\,k\in\mathbb Z^{\geq 0}\right\}$ where $a\in\mathbb N$ and $b\in\mathbb Z^{\geq 0}$.

\begin{theorem}\label{thm:mixed}
Let $\g$ be a nilpotent $\Z$-Lie algebra which is finitely generated as an abelian group. 
Then there exist  
\begin{itemize}
\item[(a)]
a constant $M=M(\g)$ only depending on $\g$, 

\item[(b)] a partition $\left\{\mathscr{P}_1, \dots, \mathscr{P}_r\right\}$ of the set of primes larger than $M$ into Frobenius sets, 
\item[(c)] a partition $\{\mathscr{F}_1, \dots, \mathscr{F}_s\}$ of natural numbers into  arithmetic progressions, and 
 \item[(d)]  polynomials $g_{ij}(T)$, for  $1\leq i\leq r$ and $1\leq j\leq s$, with non-negative integer coefficients,  
 \end{itemize} such that 
for all $q:=p^f$ where $(p,f) \in \mathscr{P}_i \times \mathscr{F}_j$
 for some
$ 1 \le i \le r$ and $1\le j\le s$, we have
\[
\mf(\GG_{\F_{q}})=fg_{ij}(q). 
\]
\end{theorem}

We remark that the groups $\GG_{\F_q}$ have exponent $p$, so that  by a result of Brauer  $\Q(e^{2\pi i/p})$ is a splitting field of $\GG_{\F_q}$. It follows from~\cite{Merkurjev} that $\mathrm{ed}_\mathbb K(\GG_{\F_q})=\mf(\GG_{\F_q})$ for every field $\mathbb K$ that contains $\Q(e^{2\pi i/p})$.

A natural extension of the class of  $p$-groups  $\GG_{\F_q}$ is the groups $\GG_R$ with $R\cong \mathcal O/\mathfrak p^d$, where $\mathcal O$ is the valuation ring of a local field and $\mathfrak p$ is the prime ideal of $\mathcal O$. 
Up to isomorphism, every finite local ring with a nilpotent and principal maximal ideal is of the form $\mathcal O/\mathfrak p^d$ (see~\cite{Deligne,McLean}). Such rings are called \emph{finite truncated valuation 
rings}.

Let $R$ be  a finite truncated valuation ring with maximal ideal $\mathfrak m$. We use $d=d_R$ to denote the  smallest positive integer satisfying $\mathfrak m^d=0$. If $p=p_R:=\mathrm{char}(R/\mathfrak m)$, then $pR=\mathfrak m^e$ for some $e=e_R\geq 1$ and $\mathcal O/\mathfrak m\cong \F_{p^f}$ for some $f=f_R\geq 1$. 
Henceforth we refer to the quadruple $(d,p,e,f)$ as the \emph{associated parameters} of $R$.  It turns out that we can formulate the problem (and our results) only in terms of  $(d,p,e,f)$, even though these parameters do not identify $R$ uniquely up to isomorphism.
We conjecture the following extension of Theorem~\ref{thm:mixed}
for  $\mf(\GG_R)$.
\begin{conjecture}
\label{conjj}
Let $\g$ be as in Theorem~\ref{thm:mixed}. Then the
constant $M(\g)$, the  partitions $\{\mathscr{P}_1,\ldots,\mathscr P_r\}$ and 
$\{\mathscr{F}_1,\ldots,\mathscr F_s\}$,  and the polynomials $g_{ij}(T)$ can be chosen such that
for every finite truncated valuation ring $R$ with associated parameters $(d,p,e,f)$ we have
\[
\mf(\GG_R)=f\sum_{\ell=0}^{e-1}g_{ij}(p^{f(d-\ell)})
\quad\text{ whenever }(p,f)\in \mathscr P_i\times \mathscr F_j.\] 
\end{conjecture}
We are unable to prove Conjecture~\ref{conjj} in full generality, but
we
establish two results 
that support this conjecture. The first result, Theorem~\ref{thm:upperbound}, shows that $\mf(\GG_R)$ is not greater than the conjectural value, but grows with the same rate as $d$ tends to infinity.

\begin{theorem}
\label{thm:upperbound}
Let $\g$, $\{\mathscr P_i\}_{i=1}^r$, 
$\{\mathscr F_i\}_{i=1}^s$, and the $g_{ij}(T)$ 
be as in 
Theorem~\ref{thm:mixed}. 
Fix 
$1\leq i_\circ\leq r$ and $1\leq j_\circ\leq s$.
Then there exists a constant $M>0$ such that for 
every finite truncated valuation ring $R$ with associated parameters $(d,p,e,f)$
satisfying $(p,f)\in\mathscr P_{i_\circ}\times \mathscr F_{j_\circ}$ and  $p>M$, we have   
\[
fg^{}_{i_\circ j_\circ}(p^{fd})
\leq \mf(\GG_{R})\leq f\sum_{\ell=0}^{e-1} g_{i_\circ j_\circ}^{}(p^{f(d-\ell)}).
\]
\end{theorem}

\begin{remark}
Note that the above theorem establishes Conjecture~\ref{conjj} in the case $e=1$. In this case, the rings $R$  are of the form $\mathcal O/\mathfrak p^d$ where $\mathcal O$ is the ring of integers of a finite unramified extension of $\Q_p$. 
\end{remark}

The second result, Theorem~\ref{partial-order-Lie}, establishes the conjecture for groups associated to \emph{pattern Lie algebras}.
Before stating Theorem~\ref{partial-order-Lie} 
we recall some definitions. 
Let $\prec$ be a partial order on the set $[n]:=\{1, 2,\dots , n\}$. To this partial order, we assign the $\Z$-Lie 
subalgebra $\g_\prec$ of $ \mathfrak{gl}(n,\Z)$ that is spanned by the $e_{ij}$ satisfying $i\prec j$, where the $e_{ij}$ denote the standard basis vectors of $ \mathfrak{gl}(n,\Z)$. 
Many nilpotent Lie algebras of interest (e.g.,~nilradicals of parabolic subalgebras of semisimple Lie algebras) are pattern Lie algebras. 

Without loss of generality we assume that $\prec$ is compatible with the usual ordering of $[n]$, so that  elements of $\g_\prec$  are strictly upper triangular matrices.   For any $i\prec j$ define 
\begin{equation}
\label{eq:alphai,j}
\alpha(i,j):=\#\{k\in [n]: i\prec k\prec j\}.
\end{equation}
A pair $(i,j)$ with $i\prec j$ is called \emph{extreme} if $i$ is minimal and $j$ is maximal with respect to $\prec$. The set of extreme pairs will be denoted by 
$\Iex$. \begin{theorem}\label{partial-order-Lie}
Let $\g:=\g_\prec$ with $\g_\prec$ as above, and let $R$ be a finite truncated valuation ring with associated parameters $(d,p,e,f)$ such that  
\[
p>\max\{\alpha(i,j): i\prec j\}+1.
\]
Then
\begin{equation}
\mf(\GG_R)=\sum_{\ell=0}^{e-1 }\sum_{(i,j)\in \Iex} fp^{f(d-\ell)\alpha(i,j)}.
\end{equation}
\end{theorem}
Special cases of Theorem~\ref{partial-order-Lie} were proved in~\cite{BMS19} (for $R=\F_q$) and~\cite{BMS} (for Heisenberg and unitriangular groups over general $R$). 
The Heisenberg Lie algebra (which is the nilradical of the parabolic subalgebra of $\mathfrak{gl}_n$ with Levi factor $\mathfrak{gl}_1\oplus \mathfrak{gl}_{n-2}\oplus \mathfrak{gl}_1$)  corresponds to the partial order
\[
1\prec2,3,\ldots,n-1\prec n
,\]
and the unitriangular Lie algebra corresponds to the total order $1\prec\cdots\prec n$.

We conclude this paper by computing $\mf(\GG_R)$ when $\g\cong \m_{n,c}$, where $\m_{n,c}$ is the free metabelian nilpotent Lie algebra of class $c$ on $n$ generators. 
Recall that
\[
\m_{n,c}:= \f_{n,c}/ [ [\f_{n,c}, \f_{n,c}], [\f_{n,c}, \f_{n,c}]],
\]
where $\f_{n,c}$ denotes the free nilpotent $\Z$-Lie algebra of class $c$ on $n$ generators.
Indeed $\m_{n,c}$ is the largest metabelian quotient of 
$\f_{n,c}$. The following theorem is a generalization of 
\cite[Theorem 2.15]{BMS}, which corresponds to the special case $\m_{2,c}$. Special cases of $\m_{3,5}, \m_{3,6}$ and $\m_{n,4}$ for all $n \ge 2$ over finite fields were computed by Tielker \cite{Tielker}.

\begin{theorem}\label{meta}
Set $\g :=\m_{n,c}$ where $n,c \ge 2$. Then for $p$ sufficiently large we have   
\[
\mf(\GG_R)=
(c-1) {n+c-2 \choose c}
\sum_{\ell=0}^{e-1}
  fp^{f(d-\ell)}.
\]

\end{theorem}

Let us now explain the key ideas that are used in the present paper. 
Similar to our previous work \cite{BMS19}, we heavily use the Kirillov orbit method for finite $p$-groups, as well as  the notion of the commutator matrix associated to a nilpotent Lie algebra, defined originally by Grunewald and Segal \cite{Grunwald-Segal}. Another important tool is a theorem from \cite{BMS19} which relates the faithful dimension to the question of existence of rational points in general position on certain rank varieties. However, the proof of  Theorem~\ref{thm:mixed} requires two new ideas: first, a theorem of Ax from his seminal work \cite{Ax2} (see also \cite{vandenDries}) which roughly states that any elementary statement in the language of finite fields is equivalent to one about single-variable polynomials; second, a theorem of Dedekind about factorization of the reduction modulo $p$ of a monic polynomial with integer coefficients.
In addition, in the proof of Theorem~\ref{thm:upperbound} we need
a (not so well known) variant  of the multivariate Hensel's lemma
that does not require  any smoothness or genericity assumption (see Proposition~\ref{prop:filterhensel}).
This variant was first established by Ax and Kochen~\cite{AxKochenI} using ultrafilter methods (for further references, see Section~\ref{Sec:up}).

\subsection*{Acknowledgement}
The authors thank Emmanuel Breuillard, Jamshid Derakhshan, Julia Gordon, Zinovy Reichstein, and Christopher Voll
for stimulating conversations, and the anonymous referees for reading the paper carefully. The research of Hadi Salmasian
was partially supported 
by an NSERC Discovery Grant (RGPIN-2018-04044).

\section{Polynomiality along all finite fields}
\label{Sec:polyfinfil}
In this section we will prove Theorem \ref{thm:mixed}. 
 We begin with recalling a result of  Ax from model theory of finite fields.
Let ${\mathcal{Q}}:=\{p^f\,:\,p\text{ is prime and }f\in\mathbb N\}$ be the set of all prime powers. 
For $n \ge 1$, denote by ${\mathcal{B}_n}$ the Boolean algebra on ${\mathcal{Q}}$ generated by the 
subsets  
\[Z_{h_1,\ldots,h_k}:= \{ q\in\mathcal Q:  
\text{ the equations }h_i(T_1,\ldots,T_n)=0, 1\leq i\leq k, \text{ have a solution in }\F_q \},  \]
where $h_1,\ldots,h_k$ range over all polynomials in $\Z[T_1, \dots, T_n]$.
Also, let $\mathcal B$ be the Boolean algebra generated by 
finite subsets of $\mathcal Q$
and 
the sets $Z_h$ where $h\in\Z[T_1]$. 
In \cite[Theorem 11]{Ax2}, Ax proved that for all $n \ge 1$, we have ${\mathcal{B}_n}\subseteq \mathcal{B}$. This statement generalizes~\cite[Theorem 1]{Ax}, which proves a similar assertion for primes. 
It also follows from a result of van den Dries 
in~\cite[Subsec. (2.3)]{vandenDries}. 

In order to prove 
Theorem~\ref{thm:mixed} we need
 a refinement of the aforementioned result of Ax (see Proposition~\ref{lemma:mixed}). This refinement may be well known to the experts in the field (see e.g.~ \cite[Theorem 11']{Ax2} for a somewhat similar statement), but we were unable to locate the required formulation  in the literature. For this reason and for convenience of the reader we will supply a proof. We would like to thank  J. Derakhshan for communicating the main steps of this proof to us.

 Let us recall some basic facts from algebraic number theory. 
Let $h(T)\in\Z[T]$ be a monic  polynomial. 
The discriminant of $h(T)$ is defined to be
$$
\mathrm{Disc}_h:=\prod_{1\leq i<j\leq n}(\alpha_i-\alpha_j)^2,$$
where $\alpha_1,\dots,\alpha_n$ are the roots of $h(T)$. Note that since $h(T)$ is monic, $\mathrm{Disc}_h \in \Z$. 
We denote the splitting field of $h(T)$ by $E$, that is, $E:=\Q(\alpha_1,\dots,\alpha_n)$, and  the ring of algebraic integers of $E$ by $\mathcal O_E$.

Let $p$ be a prime number that does not divide $\mathrm{Disc}_h$.  Then $p$ is unramified in every simple extension $\Q(\alpha_i)$ (see \cite[Theorem 3.24]{Marcus}), hence by \cite[Theorem 4.31]{Marcus} it is also unramified in the compositum of the $\Q(\alpha_i)$, that is, in $E$. 
Therefore to every prime ideal $\mathfrak q$ of $\mathcal O_E$ that lies over $p\Z$ (i.e.,~$\mathfrak q\cap \Z=p\Z$) one can associate a Frobenius automorphism 
\cite[Theorem 4.32]{Marcus}, 
 which is the unique element $\sigma_{\mathfrak q}\in \mathrm{Gal}(E/\Q)$
that satisfies
\[
\sigma_\mathfrak{q} (\alpha)\equiv \alpha^p \mod\mathfrak{q}\quad\text{ for all }\alpha\in \mathcal O_E.
\]
The $\sigma_\mathfrak{q}$  form a conjugacy class of $\mathrm{Gal}(E/\Q)$, 
denoted by the Artin symbol $\left( \frac{E/\Q}{p} \right)$. 

Denote the reduction of $h(T)$ modulo $p$ by $\bar{h}(T)$. Since $p\nmid\mathrm{Disc}_h$,
the roots of $\bar{h}(T)$ are also simple. 
For the proof of the next theorem, 
see~\cite[Theorems 4.37 and 4.38]{Jacobson}.
\begin{theorem}[Dedekind]\label{Dedekind}
Let $h(T)$ and $p$ be as above. Suppose 
\begin{equation}\label{dedekind}
\bar{h}(T)=h_1(T)\cdots h_r(T)
\end{equation}
is the factorization of $\bar{h}(T)$ into irreducible polynomials over $\mathbb{F}_p$. Set $n_i:=\deg h_i$ for $1\leq i\leq r$. Let  $\mathfrak{q}$ be a prime ideal of $\mathcal O_E$ that lies over $p\Z$.
Then the cycle decomposition of $\sigma_\mathfrak{q}$, when viewed as a permutation of the roots of $h(T)$, is of the form $\varsigma_1\cdots \varsigma_r$ with each $\varsigma_i$ a cycle of length $n_i$.
\end{theorem}
\begin{remark}
Theorem~\ref{Dedekind} is not stated in~\cite{Jacobson} exactly in the given form. 
Jacobson proves (by an argument originally due to Tate) 
that for every prime ideal $\mathfrak{p}$ of $\mathcal O_h:=\Z[\alpha_1,\ldots,\alpha_n]$ that lies over $p\Z$, there exists a unique $\sigma_{\mathfrak{p}}\in\mathrm{Gal}(E/\Q)$ such that for all $\alpha\in \mathcal O_h$ one 
has $
\sigma_\mathfrak p(\alpha)\equiv\alpha^p\mod\mathfrak{p}
$, 
and furthermore the restriction of $\sigma_\mathfrak{p}$ to the roots of $h(T)$ has the desired cycle decomposition.  Now if $\mathfrak q\subset \mathcal O_E$ is a prime ideal that lies over $p\Z$, then 
$\sigma_\mathfrak{q}
$ also satisfies
$
\sigma_\mathfrak q(\alpha)\equiv\alpha^p\mod\mathfrak{p}
$ for $\alpha\in \mathcal O_h$, where $\mathfrak p:=\mathfrak q\cap \mathcal O_h$. Thus by the uniqueness of $\sigma_\mathfrak{p}$ we obtain $\sigma_{\mathfrak p}=\sigma_{\mathfrak q}$.

\end{remark}

\begin{lemma}
\label{lem:Sp(h)=...}Let $h(T)$, $p$, $\sigma_\mathfrak q$, and $n_i$ for $1\leq i\leq r$ be as in Theorem~\ref{Dedekind}. 
Denote by $S_p(h)$ the set of integers $f \ge 1$ such that $ \bar{h}(T)$ has a root in $ \mathbb{F}_{p^f}$. 
Then
\[
S_p(h)= 
\bigcup_{ i=1}^{  r } n_i \mathbb N=
\{ k \ge 1: \fix( \sigma_\mathfrak{q}^k) \neq \varnothing \},
\] 
where $\fix( \sigma_\mathfrak{q}^k)$ is the set of fixed points of 
$\sigma_\mathfrak{q}^k$, viewed as
a permutation  of the roots of $h(T)$. 
\end{lemma}

\begin{proof}
From the cycle decomposition of $\sigma_\mathfrak{q}$ in Theorem~\ref{Dedekind} it follows immediately that  
$\sigma_\mathfrak{q}^k$ has a fixed point if and only if $n_i|k$ for some $1\leq i\leq r$. This proves the second equality. 
Thus it remains to verify  that $S_p(h)=\bigcup_{i=1}^r n_i\mathbb N$. 

First observe that for each $ 1 \le i \le r$, 
the quotient  $ \mathbb{F}_p[x]/ ( h_i(x) )$ is a field of order $p^{n_i}$ in which  $\bar h(x)$ has a root. 
Recall that  $k|l$ if and only if  
$\mathbb{F}_{p^l}$ has a subfield  isomorphic to 
$\mathbb{F}_{p^k}$. Thus if $k \in n_i \mathbb N$ for some $1 \le i \le r$, then $\bar h(x)$ has a root in $\mathbb{F}_{p^k}$. This proves $S_p(h)\supseteq\bigcup_{i=1}^r n_i\mathbb N$. 
For the reverse inclusion, assume that $ \bar h(x)$ has a root in $\mathbb{F}_{p^k}$. It follows that there exists $ 1 \le i \le r$ such that $h_i(x)$ has a root in
$\mathbb{F}_{p^k}$. Call this root $ \alpha$. Since $h_i(x)$ is irreducible over $ \mathbb{F}_p$, it follows that the subfield $\mathbb{F}_p(  \alpha)$ has $p^{n_i}$ elements. Thus $\F_{p^{n_i}}$ is isomorphic to a subfield of $\mathbb{F}_{p^k}$, and consequently $n_i$ divides $k$. 
\end{proof}

Define an equivalence relation $\sim$ on $\mathrm{Gal}(E/\Q)$ by
\[
\tau_1\ \sim\ \tau_2\text{ iff }\tau_1=\sigma\tau_2^j\sigma^{-1}
\text{ for some $\sigma\in\mathrm{Gal}(E/\Q)$,}
\]
where $\mathrm{gcd}(j,\mathrm{ord}(\tau_2))=1$. An equivalence class of $\sim$ is called a division, or an {\it Abteilung}.  
Following \cite{Lagarias}, for any division $C\subseteq\mathrm{Gal}(E/\Q)$ we consider the set
\[ \E(C) := \left\{  p \ :\    \left( \frac{E/\Q}{p} \right)  \subseteq C \right\}.
\]
In the following proposition, 
let $S(h;M)$ for $M>0$ and a monic polynomial $h\in\Z[T]$  denote the set of pairs $(p,f)$ where $p>M$ is prime, $f\geq 1$, and 
the reduction modulo $p$ of  $h(T)$ has a root 
in  $\F_{p^f}$.

\begin{proposition}\label{lemma:mixed}
Let $h(T)\in\Z[T]$ be monic. 
Then there exists $M_\circ>0$, only depending on $h(T)$, such that  $S(h;M_\circ)=\bigcup_{i=1}^r\mathcal P_i\times \mathcal N_i$ for some $r\in\mathbb N$, where the $\mathcal P_i$ and the $\mathcal N_i$ satisfy the following properties:
\begin{itemize}
\item[(a)]
The $\mathcal P_i$ form a partition of the set of primes bigger than $M_\circ$ into Frobenius sets, and  
\item[(b)]
Each $\mathcal N_i$  is a finite union of sets of the form $k\mathbb N$ for $k\in\mathbb N$. 
\end{itemize} 

\end{proposition}

\begin{proof}
Without loss of generality we can assume that $h(T)$ is irreducible over $\Q$.
As before let $E$ denote the splitting field of $h(T)$. We denote  the divisions of $\mathrm{Gal}(E/\Q)$ by  $C_1, \dots, C_r$. 
We set $M_\circ:=1+|\mathrm{Disc}_h|$. First we show that 
if $ p, p'>M_\circ $ are primes such that \[
\left( \frac{E/\Q}{p} \right)\cup\left( \frac{E/\Q}{p'} \right) 
\subseteq C_i
\text{ for some $ 1 \le i \le r$,}
\]
 then  $S_p(h)= S_{p'}(h)$. 
Choose $\sigma_\mathfrak q\in \left( \frac{E/\Q}{p} \right)$ and $\sigma_{\mathfrak{q}'}\in \left( \frac{E/\Q}{p'} \right)$. 
Then  
$\sigma_{\mathfrak{q}'}= \tau \sigma_{\mathfrak{q}}^j \tau^{-1}$, where
$\mathrm{gcd}(j,\mathrm{ord}(\sigma_{\mathfrak{q} }))=1$. 
By Lemma~\ref{lem:Sp(h)=...}, if
 $ n \in S_{p}(h)$ then 
$\fix(\sigma_{{\mathfrak{q}} }^n)\neq\varnothing$, hence also $\fix(\sigma_{{\mathfrak{q}} }^{nj})\neq\varnothing$. Since $\sigma_{\mathfrak{q}'}^n$ is conjugate to $\sigma_{{\mathfrak{q}} }^{nj}$, it follows that $\fix(\sigma_{\mathfrak{q}'}^n)\neq \varnothing$, so that $n\in S_{p'}(h)$. 
This proves $S_p(h)\subseteq S_{p'}(h)$, and by symmetry we have the reverse inclusion, hence $S_p(h)=S_{p'}(h)$.

From what we just proved it follows that $S_p(h)$ depends only on the index $1 \le i\le r$ for which $\left( \frac{E/\Q}{p} \right) \subseteq C_i$. 
It is shown in \cite[Theorem 1.1]{Lagarias} that every set of primes that
differs by only a finite set from
 $\E(C)$ for some division $C$ is a Frobenius set.\footnote{We remark that Lagarias uses different terminology. In particular, 
in \cite{Lagarias}
our Frobenius sets are called SPC sets, the sets $\E(C)$ are called elementary Frobenius sets, and the sets in the Boolean algebra generated by the $\E(C)$ for divisions $C$ are called Frobenius sets.} Thus we can set $\mathcal P_i:=\E(C_i)$ for $1\leq i\leq r$. 
The existence of the $\mathcal N_i$ now follows from  Lemma~\ref{lem:Sp(h)=...}.
\end{proof}
\begin{remark}
\label{rmk:QMbij}
Let $\mathcal Q_1\subseteq \mathcal Q$ denote the set of prime numbers. Then $S(h,M)$ is the inverse image of $Z_h\cap \mathcal Q^M$, where $Q^M:=\{p^f\in\mathcal Q\,:\,p>M\}$, under the bijection $\mathcal Q_1\times \mathbb N\to \mathcal Q$ given by $(p,f)\mapsto p^f$.
\end{remark}

For the proof of Theorem~\ref{thm:mixed} we need one more ingredient: a result from our previous work~\cite{BMS19} which relies on the Kirillov orbit method for finite 
$p$-groups. To state this result, we need to recall the notion of the \emph{commutator matrix} of a nilpotent $\Z$-Lie algebra. 

Given
a finitely generated abelian group $\Gamma$,
we call a subset $S\subseteq \Gamma$  a \emph{semibasis} if it represents a basis over $\Z$ of the free abelian group 
$\Gamma/\Gamma_\mathrm{tor}$, where $\Gamma_\mathrm{tor}$ denotes the subgroup of torsion elements of $\Gamma$. Clearly $\#S=\rank_\Z\Gamma$, where $\rank_\Z\Gamma$ is the rank of $\Gamma$ as a $\Z$-module. 

Now let $\g$ be a nilpotent $\Z$-Lie algebra which is finitely generated as an abelian group. Let $\ww_1,\ldots,\ww_{l_1}$ be a semibasis of $\ZZ(\g)\cap[\g ,\g]$. Let  $\ww_{l_1+1},\ldots,\ww_{m}$ be
 elements of 
$\g $ that represent a semibasis of 
$[\g,\g]/(\ZZ(\g)\cap [\g,\g])$. Finally, let $\zz_1,\ldots,\zz_{l_2}$ be elements of $\g$ that represent a semibasis of $\ZZ(\g)/(\ZZ(\g)\cap [\g,\g])$.
Then the vectors
$\{\ww_1,\ldots,\ww_m,\zz_1,\ldots,\zz_{l_2}\}$ are $\Z$-linearly independent and  $\rank_\Z(\ZZ(\g)+[\g,\g])=l_2+m$. 

Let $\vv'_1,\ldots,\vv'_n$ be elements of $\g$ that represent a semibasis of $\g/\ZZ(\g)$. 
For $1\leq i<j\leq n$, there exist integers $\eta_{ij}^k$, $l_1+1\leq k\leq m$, such that
if we set \[
\vv'_{i,j}:=\left(
[\vv'_i, \vv'_j]-\sum_{k=l_1+1}^m\eta^k_{ij} \ww_k
\right),
\]
then 
$\vv'_{i,j}+
([\g,\g]\cap \ZZ(\g))
\in([\g,\g]/[\g,\g]\cap \ZZ(\g))_\mathrm{tor}$. 
Set $K$ 
equal to 
 the exponent of the group
$\left(
[\g,\g]/[\g,\g]\cap\ZZ(\g)\right)_\mathrm{tor}$.  Then
$K\vv'_{i,j}\in[\g,\g]\cap\ZZ(\g)$ for  $1\leq i<j\leq n$. Now set
$\vv_i:=K\vv'_i$ for $1\leq i\leq n$. Then 
there exist integers $\lambda_{ij}^k$ such that
\begin{equation*}
[\vv_i, \vv_j]-\sum_{k=1}^m\lambda^k_{ij} \ww_k\in
\big([\g,\g]\cap\ZZ(\g)\big)_\mathrm{tor}
\text{ for  $1\leq i<j\leq n$}.
\end{equation*}
We remark that $\lambda_{ij}^k=K^2\eta_{ij}^k$ for
$l_1+1\leq k\leq m$. We set $\lambda_{ii}^k:=0$, and for $i>j$ we define $\lambda_{ij}^k:=-\lambda_{ji}^k$. 
Finally, for $1\leq i, j\leq n$ we define 
the linear forms
\begin{equation*}
\Lambda_{ij}(T_1,\dots,T_m):= \sum_{k=1}^m \lambda_{ij}^k T_k\in \Z[T_1,\dots,T_m]. 
\end{equation*}
It is clear that $\Lambda_{ii}=0$ and $\Lambda_{ij}=-\Lambda_{ji}$ for $1\leq i,j \le n$. The {\it commutator matrix of $\g$}
(relative to the chosen ordered basis) is the $n\times n$ skew-symmetric matrix of linear forms defined by 
\begin{equation}\label{commutating-matrix}
F_\g(T_1,\dots, T_m):=[ \Lambda_{ij}(T_1, \dots, T_m)]_{1 \le i,j \le n }
. 
\end{equation}
See~\cite{Grunwald-Segal,Voll04,Voll05,Voll,Avni,Stasinski-Voll} as other references for $F_\g$.
In the following theorem, $\rank_{\F_q}(\cdot)$ denotes the rank of a matrix with entries in the field $\F_q$.

For a given $\g$ we set $\mathcal{SB}_\g:=(\mathbf w_1,\ldots,\mathbf w_m,\mathbf z_1,\ldots,\mathbf z_{l_2},\mathbf v'_1,\ldots,\mathbf v'_n)$. This $(m+l_2+n)$-tuple of vectors is the semibasis data that uniquely determines $F_\g$. For an isomorphism class of $\g$ we can fix this semibasis data, thereby fixing the commutator matrix $F_\g$ whenever needed.

\begin{theorem}
\label{Algorithm}
Let $\g$ be a nilpotent $\Z$-Lie algebra which is finitely generated as an abelian group.  
Then there exists a constant $M>0$, depending only on $\g$  and the choice of $\mathcal{SB}_\g$, such that for every prime power $q:=p^f$ with  $p>M$ and $f\ge 1$ we have 
\begin{equation}
\label{eq:mfaithAlg}
\mf(\GG_q)=\min\left\{ 
\sum_{\ell=1}^{l_1} fq^{\frac{\rank_{\F_q}(F_\g(x_{\ell 1},\dots,x_{\ell m}))}{2}}: \begin{pmatrix}
x_{11} & \dots & x_{1l_1}\\
\vdots  & \ddots & \vdots\\
 x_{l_1 1} & \dots & x_{l_1 l_1}
\end{pmatrix}\in\mathrm{GL}_{l_1}(\F_q)\right\}+fl_2,
\end{equation} 
where $m:=\rank_\Z([\g,\g])$, 
$l_1:=\rank_\Z([\g,\g]\cap\ZZ(\g))$ and $l_2:=\rank_\Z(\ZZ(\g)/\ZZ(\g)\cap[\g,\g])$. 
\end{theorem}
\begin{proof}
This is proved in~\cite[Theorem 5.3]{BMS19}.
   \end{proof}

\begin{remark}
\label{rmk-boolean-prop}
(a)  Let $X\subseteq \mathbb N$ belong to the Boolean algebra generated by arithmetic progressions. Then $X$ can be expressed as a disjoint union of arithmetic progressions (note that any singleton $\{a\}\subset\mathbb N$ is also considered an arithmetic progression). This can be seen as follows. Set $A_{a,b}:=\{a+kb\,:\,k\in\Z^{\geq 0}\}$ where $a\in\mathbb N$ and $b\in\mathbb Z^{\geq 0}$. Then
 $A_{a,b}\cap A_{a',b'}$ is either the empty set, or a singleton, or a set of the form $A_{a'',b''}$ where $b''$ is the least common multiple of $b$ and $b'$. 
Furthermore, \[
 \mathbb N\setminus A_{a,b}=\bigcup_{1\leq a'\neq a\leq b}A_{a',b}\cup\bigcup_{\tiny\begin{array}{c}a''\equiv a\text{ mod }b\\
1\leq a''<a \end{array}}\{a''\}.\] Finally, by De Morgan's Law the case of  $A_{a,b}\cup A_{a',b'}$ can be reduced to the cases of intersection and complement. From these observations it follows that every subset of $\mathbb N$ that is obtained by finite intersection, complements, and finite unions is expressible as a disjoint union of the sets $A_{a,b}$.

(b) 
Recall that $\mathcal Q_1\subseteq \mathcal Q$ denotes the set of prime numbers. Suppose that  $X\subseteq \mathcal Q_1\times \mathbb N$ belongs to the Boolean algebra generated by product sets  $\mathcal P\times \mathcal F$ where $\mathcal P\subseteq \mathcal Q_1$ is a Frobenius set and $\mathcal F\subseteq\mathbb N$ is in the Boolean algebra generated by arithmetic progressions.
Then $X$ can be expressed as a union of the generators $\mathcal P\times \mathcal F$. This is a consequence of part (a) and the following relations:
\begin{itemize}
\item[(i)]
$(\mathcal P\times\mathcal F)\cap (\mathcal P'\times \mathcal F')=(\mathcal P\cap\mathcal P')\times (\mathcal F\cap \mathcal F')$.
\item[(ii)] $(\mathcal Q_1\times\mathbb N)\setminus(\mathcal P\times \mathcal F)=((\mathcal Q_1\setminus \mathcal P)\times \mathcal F)\cup
(\mathcal Q_1\times (\mathbb N\setminus \mathcal F))\cup ((\mathcal Q_1\setminus \mathcal P)\times (\mathbb N\setminus \mathcal F))$.  
\end{itemize}
Note that as in (a), the union $(\mathcal P\times \mathcal F)\cup (\mathcal P'\times \mathcal F')$
can be expressed by intersection and complement.
\end{remark}

\begin{lemma}
\label{lem:XxY}
Let $\mathscr X$ and $\mathscr Y$ be two arbitrary sets, and let
$\{\mathscr Z_i\}_{i\in \mathcal I}$ be a partition of $\mathscr X\times \mathscr Y$ into 
finitely many sets such that $\mathscr Z_i=\bigcup_{j=1}^{N_i} X_{i,j}\times Y_{i,j}$ for $i\in\mathcal I$, where $N_i\in \mathbb N$. Then there exist partitions 
$\{X_k\}_{k\in \mathcal I_X}$ of $\mathscr X$ and 
$\{Y_l\}_{l\in \mathcal I_Y}$ of $\mathscr Y$, where $|\mathcal I_X|,|\mathcal I_Y|<\infty$,  such that the following hold.
\begin{itemize}
\item[\rm (a)] The $X_k$ belong to the Boolean algebra generated by the $X_{i,j}$. 
\item[\rm (b)] The $Y_l$ belong to the Boolean algebra generated by the $Y_{i,j}$.
\item[\rm (c)]
Each $\mathscr Z_i$ is a union of Cartesian product sets of the form  $X_k\times Y_l$.
\end{itemize}
\end{lemma}

\begin{proof}
Every finite Boolean algebra is isomorphic to the power set of a finite set. From this statement it follows that   there exists a finite family of  disjoint sets 
$\mathscr F_\mathscr X:=\{X_k\}_{k\in\mathcal I_X}$  in 
 the Boolean algebra generated by the $X_{i,j}$ such that each $X_{i,j}$ can be expressed as the union of elements of a subset of  
$\mathscr F_{\mathscr X}$. Similarly, there   
exists a finite family of  disjoint sets 
$\mathscr F_\mathscr Y:=\{Y_l\}_{l\in\mathcal I_Y}$  in the  Boolean algebra generated by the $Y_{i,j}$ such that each $Y_{i,j}$ can be expressed as the union of elements of a subset of  
$\mathscr F_{\mathscr Y}$. It is clear that each $\mathscr Z_i$ is a union of the sets $X_k\times Y_l$. Since $\bigcup_{i\in\mathcal I}\mathscr Z_i=\mathscr X\times \mathscr Y$, we have $\bigcup_{i\in\mathcal I}\bigcup_{1\leq j\leq N_i}X_{i,j}=\mathscr X$ and $\bigcup_{i\in\mathcal I}\bigcup_{1\leq j\leq N_i}Y_{i,j}=\mathscr Y$, hence $\bigcup_{k\in\mathcal I_X}X_k=\mathscr X$ and $\bigcup_{l\in \mathcal I_Y}Y_k=\mathscr Y$. 
\end{proof}

\begin{proof}[Proof of Theorem \ref{thm:mixed}]
Let $p$ be a prime number and set $q:=p^f$ for some $f\geq 1$.  Since $F_\g(T_1,\dots,T_m)$ is a skew-symmetric $n\times n$ matrix where $n:=\rank_\Z(\g/\ZZ(\g))$,  the rank of $F_\g(x_1,\dots,x_m)$ for $x_1,\ldots,x_m\in\F_q$ is an even number no larger than $n$. Let $\mathcal M$ be the set of all integer vectors $\mu=(a_1,\dots,a_{l_1})\in\Z^{l_1}$ satisfying $0\leq a_1\leq \cdots \leq a_{l_1}\leq n/2$. To each $\mu \in\mathcal  M$ we assign a polynomial 
$g_\mu(T)$
defined by
$$
g_\mu(T):=T^{a_1}+\dots+T^{a_{l_1}}+l_2.
$$
We order the vectors $\mu$ according to  the {\it reverse lexicographical order}  $\lhd$, defined by $\mu \lhd \mu'$ if the rightmost non-zero component of the vector $\mu'-\mu$ is positive.  If $\mu\lhd\mu'$ and $q>l_1$, then we can easily see that
\begin{equation}\label{ine-g}
g_\mu(q)<g_{\mu'}(q).
\end{equation}
We sort elements of $\mathcal M$ as 
$\mu_1\lhd\dots\lhd\mu_N$, where $N:= \# \mathcal M$.
For   $\mu=(a_1,\ldots,a_{l_1})\in\mathcal M$, we define an affine variety \begin{equation*}
{\bf X}_\mu:=\left\{(x_{ij})\in \mathrm{Mat}_{l_1\times m}(\C):  \rank_\C F_\g(x_{i1},\dots, x_{im}) =2 a_i\text{ for }
1\leq i\leq l_1,\ 
\det\left(x_{ij}\right)_{1\leq i,j\leq l_1}\neq 0
\right\}.
\end{equation*}
Note that the non-vanishing condition on the determinant can be turned into an equation by introducing a new variable, standing for the inverse of the determinant. We also remark that ${\bf X}_\mu$ is defined over $\Z$ because $F_\g(T_1,\dots,T_m)$ is an integer matrix. 
For $\mu\in \mathcal M$ set
$$
\Sigma_\mu:=\left\{p^f\,:\,p\text{ is prime},\
p>\max\{l_1,M_1\},\ f\geq1,\ \text{ and }{\bf X}_\mu(\F_{p^f})\neq \varnothing\right\},
$$
where  $M_1$ is the lower bound constant of   Theorem~\ref{Algorithm}.  
Theorem~\ref{Algorithm} and~\eqref{ine-g} imply that
$\mf(\GG_{p^f})=fg_{\mu_k}(p^f)$
for the value of $k$ that satisfies 
$$p^f\in\mathscr{Q}_k:=\Sigma_{\mu_k}\setminus \bigcup_{1\leq i<k}\Sigma_{\mu_i}.$$

Recall that
$\mathcal Q$ denotes the set of prime powers.
By Ax's theorem~\cite[Theorem 11]{Ax2} stated in the beginning of this section, 
the sets $\mathscr{Q}_k$ can be expressed in terms of unions, intersections, and complements of finitely many  of the generators of the Boolean algebra $\mathcal B$, that is,  finite subsets of $\mathcal Q$ and
the sets $Z_h$ for $h(T)\in\Z[T]$.  
Recall that for $M>0$ we set \[
\mathcal Q^{M}:=\{p^f\in\mathcal Q\,:\,p>M\}.
\] For a non-constant polynomial $h(T)\in\Z[T]$, 
there exists a monic polynomial $h_1(T)\in\Z[x]$ and a constant $M_h>0$ such that $Z_h\cap \mathcal Q^{M_h}=Z_{h_1}\cap \mathcal Q^{M_h}$ 
(we obtain $h_1$ from $h$ by a suitable change of variable $T\mapsto T/T_0$ and then clearing the denominators, so that  it suffices to choose $M_h>|T_0|$).
It follows that for 
$M'$ sufficiently large, the sets $\mathscr Q_k\cap \mathcal Q^{M'}$ are in the Boolean algebra of subsets of $\mathcal Q^{M'}$
that is generated by the sets $Z_h\cap \mathcal Q^{M'}$ for monic $h(T)\in\Z[T]$. 
Proposition~\ref{lemma:mixed}, Remark~\ref{rmk:QMbij}
and Remark~\ref{rmk-boolean-prop}(b) imply that there exists
a constant $M(\g)$ such that
each of the sets $\tilde{\mathscr{Q}}_k:=\left\{(p,f)\,:\,p^f\in\mathscr Q_k\text{ and }p>M(\g)\right\}$
for $1\leq k\leq N$ is a finite union of Cartesian products $\mathcal P\times \mathcal N$ where $\mathcal P$  is a Frobenius set and $\mathcal N$ is  an arithmetic progression.
The existence of the desired partitions $\mathscr P_1,\ldots,\mathscr P_r$ and $\mathscr F_1,\ldots,\mathscr F_s$ now follows from 
Lemma~\ref{lem:XxY} where we set $\mathscr X:=\mathcal Q^{M(\g)}$, $\mathscr Y:=\mathbb N$, and $\mathscr Z_k:=\tilde{\mathscr D}_k$ for $1\leq k\leq N$.
\end{proof}
The proof of Theorem~\ref{thm:mixed} has the following consequence, which we will need when considering $\mf(\GG_R)$ for finite truncated valuation rings $R$.

\begin{corollary}
\label{cor:mixed}
Let $\g$, $\{\mathscr P_i\}_{i=1}^r$,  
$\{\mathscr F_i\}_{i=1}^s$, 
and the $g_{ij}(T)$ 
be as in 
Theorem~\ref{thm:mixed}. Fix $1\leq i_\circ\leq r$ and $1\leq j_\circ\leq s$. Set $\mathscr P:=\mathscr P_{i_\circ}$, 
$\mathscr F:=\mathscr F_{j_\circ}$, and $g(T):=g_{i_\circ j_\circ}(T)$. 
Then $g(T)=l_2+\hat{g}(T)$, where $\hat{g}(T)\in\Z[T]$, such that the following statements hold.
\begin{itemize}
\item[(a)]

$\hat{g}(T)=\sum_{i=0}^{\deg(\hat{g})}\hat{g}_iT^i$ where the $\hat{g}_i$ are non-negative integers and
 $\sum_{i=0}^{\deg(\hat{g})}\hat{g}_i=l_1$.\item[(b)]
For any $q:=p^f$ where $(p,f)\in\mathscr P\times \mathscr F$, there exist $l_1$ vectors  $\mathbf a_i\in\F_q^{\oplus m}$, $1\leq i\leq l_1$, such that their projections on the first $l_1$ components form an $\F_q$-basis of $\F_q^{\oplus l_1}$ and 
\[
\#\left\{i\,:\,
1\leq i\leq l_1\text{ and }\mathrm{rk}_{\F_q}(F_\mathfrak g(\mathbf a_i))=2j
\right\}=\hat{g}_j\quad 
\text{for }0\leq j\leq \deg(\hat{g}).
\]  
\end{itemize}
\end{corollary}
\begin{proof}
From the right hand side of~\eqref{eq:mfaithAlg}
it is clear that $g(T)-l_2$ is expressible as a sum of $l_1$ powers of $q$ with non-negative coefficients. Indeed the coefficient $\hat{g}_i$ is equal to the number of summands $q^i$. This proves (a). The minimum value on the right hand side of~\eqref{eq:mfaithAlg} corresponds to a choice of 
$l_1$ vectors $\mathbf a_1,\ldots,\mathbf a_{l_1}$ that satisfy the assumptions of (b).   
\end{proof}

\section{Bounds for $\mf(\GG_R)$ }
\label{Sec:up}
In this section we turn our attention to $\mf(\GG_R)$, where $R$ is a finite truncated valuation ring. It is known (see for example \cite{McLean}) that any such   $R$ is isomorphic to $\mathcal O/\mathfrak p^d$ for some $d\geq 1$, where 
$\mathcal O$ is the ring of integers of a $p$-adic field and $\mathfrak p$ is the maximal ideal of $\mathcal O$
(thus, in the notation of Section~\ref{sec:intro} we have $\mathfrak{m}=\mathfrak p/\mathfrak p^d$). Henceforth we assume that $R=\mathcal O/\mathfrak p^d$. 

 By analogy with the special case where $R$ is a finite field, one  expects that a suitable variation of  Theorem~\ref{Algorithm} holds for $\mf(\GG_R)$. Indeed the method of proof of Theorem~\ref{Algorithm} can be adapted to prove Proposition~\ref{Algorithm-R} below.
Recall that $p=\mathrm{char}(\mathcal O/\mathfrak p)$, 
$\#\mathcal O/\mathfrak p=p^f$, and  
{$p(\OO{d})=\mathfrak p^e/\mathfrak p^d$}. It follows that $\mathcal O/\mathfrak p^e$ is an $\F_p$-vector space.
We define the map 
\begin{equation}
\label{eq:dfnofprj}
\mathsf{proj}:(\mathcal O/\mathfrak p^d)^{\oplus m}\to (\mathcal O/\mathfrak p^e)^{\oplus l_1}
\end{equation} 
to be the natural projection on the first $l_1$ components, i.e.,  \[\mathsf{proj}(a_1,\ldots,a_m):=(\bar{a}_1,\ldots,\bar{a}_{l_1}),
\] where $a\mapsto \bar a$ denotes the natural projection $\mathcal O/\mathfrak p^d\to \mathcal O/\mathfrak p^e$.

\begin{remark}
\label{rmk:exactseq}
Any finitely generated abelian group $A$ can be expressed as a direct sum $A\cong A_F\oplus A_T$ where $A_F$ is free abelian and $A_T$ is torsion. 
Tensoring $A$ (over  $\mathbb Z$) with a commutative ring $R$ whose characteristic is a prime number larger than the exponent of $A_T$ annihilates $A_T$. Furthermore, if $A$ and $B$ are finitely generated abelian groups and $A\to B$ is an injective group homomorphism, then the induced map $
A_F\to B/B_T\cong B_F
$ is also an injection. In particular a basis of $A_F$ is mapped to a linearly independent set of vectors. After tensoring with $R$, if the charatecteristic of $R$ is a sufficiently large prime, the image of the basis still remains linearly independent (because the determinant of a square minor of maximum size will remain nonzero).  It follows that  a short exact sequence of finitely generated abelian groups will remain exact upon tensoring with a finite truncated valuation ring $R$ as long as $p$ is sufficiently large (depending on the exact sequence).
\end{remark}

In the proof of Theorem \ref{thm:upperbound} we need the following extension of Theorem \ref{Algorithm}.
\begin{proposition}
\label{Algorithm-R}
Let $\g$ be as in Theorem~\ref{Algorithm} and let $R$ be a finite truncated local ring.  Let
$
\mathcal O(d,m,e,l_1)
$
denote the set of $fe l_1$-tuples $(\mathbf a_1,\ldots,\mathbf a_{fe l_1})$ 
such that $\mathbf a_i\in (\mathcal O/\mathfrak p^d)^{\oplus m}$ for all $1\leq i\leq fe l_1$, and the projections $\{\mathsf{proj}(\mathbf a_i)\}_{i=1}^{fe l_1}$ form a basis of the $\F_p$-vector space 
$(\mathcal O/\mathfrak p^e)^{\oplus l_1}$.
Then there exists a constant $M>0$, depending only on $\g$ and the choice of the semibasis data $\mathcal{SB}_\g$ that is used to define $F_\g$, such that for $p>M$ we have
\[
\mf(\GG_R)=fl_2e+
\min
\left\{
\sum_{i=1}^{fe l_1}
\left(
\frac{p^{fd(\mathrm{rk}_\Z(\g) -l_1-l_2)}}{\#\ker_{\mathcal O/\mathfrak p^d}(F_{\g}(\mathbf a_i))}
\right)^\frac12
\ :\ 
(\mathbf a_1,\ldots,\mathbf a_{fe l_1})\in \mathcal O(d,m,e,l_1)\right\}.
\]
\end{proposition}

\begin{proof}
Since the argument is nearly identical to the proof of Theorem~\ref{Algorithm}, we will only explain the points of divergence from~\cite{BMS19}. In what follows, we will refer to the notation of~\cite{BMS19}. 
We will not give an explicit value of $M$, but it  can be  computed  similar to~\cite{BMS19}.
Using Remark~\ref{rmk:exactseq} and some elementary arguments, one can verify that 
$Z(\g_R)=Z(\mathscr G_R)=Z(\g)_R$, $[\g_R,\g_R]=[\g,\g]_R$, and $Z(\g)_R\cap [\g_R,\g_R]=(Z(\g)\cap [\g,\g])_R$. 
In addition, the latter $R$-modules embed as free submodules of $\g_R$. We choose a primitive additive character $\psi:\mathcal O/\mathfrak p^d\to \C^*$  as follows.  
Let $\psi_\circ:\Q_p\to \C^*$ be the additive character defined by $\psi_\circ(x):=e^{2\pi i\{x\}_p}$, where $\{x\}_p$ denotes the $p$-adic fractional part of $x$. Note that  $\ker\psi_\circ=\Z_p$. Now let $\mathbb K$ denote the fraction field of $\mathcal O$, and set \[
\mathfrak D:=\left\{x\in \mathcal O\ :\ \mathrm{Tr}_{\mathbb K/\Q_p}(x\mathcal O)\subseteq \Z_p\right\}
.\] Then $\mathfrak D=\mathfrak p^{-\ell}$ for some $\ell\geq 0$. Fixing a uniformizer $\varpi\in\mathfrak p$, we obtain an isomorphism of the character group  $\widehat{\mathcal O/\mathfrak p^d}$ with $\mathcal O/\mathfrak p^{d}$ via the map
$
b+\mathfrak p^{d}\mapsto \psi_b\in\widehat{\mathcal O/\mathfrak p^d}
$, where 
$\psi_{b}(x):=\psi_\circ(\mathrm{Tr}_{\mathbb K/\Q_p}(\varpi^{-d-\ell}bx))
$.
Note that $\psi_{b}(x)=\psi(bx)$ where $\psi:=\psi_1$. From this  it follows that 
$\ker(\psi)$ does not contain any nontrivial ideal of $\mathcal O/\mathfrak p^d$. 
Choose $\mathbf u_1,\ldots,\mathbf u_{l_3}\in\g$ to be representatives of a semibasis of $\g/(Z(\g)+[\g,\g])$. Then by  the above remarks for $p$ sufficiently large $\g_R$ is a free $R$-module with generators $\mathbf w_1,\ldots,\mathbf w_m,\mathbf z_1,\ldots,\mathbf z_{l_2},\mathbf u_1,\ldots, \mathbf u_{l_3}$. Consequently, additive characters of $\g_R$ are of the form
\[
\psi_{\mathbf a}:=
\psi\left(\sum_{i=1}^m a_i\mathbf w_i+\sum_{i=1}^{l_2}b_i\mathbf z_i+\sum_{i=1}^{l_3}c_i\mathbf u_i\right),
\]
where $\mathbf a:=(a_1,\ldots,a_m,b_1,\ldots,b_{l_2},c_1,\ldots c_{l_3})\in R^{\oplus (m+l_2+l_3)}$.
The dimension of the irreducible representation of $\GG_R$ corresponding to $\psi_{\mathbf a}$ is 
equal to 
\[
\left(
\frac{q^{d(\mathrm{rk}_\Z(\g) -l_1-l_2)}}
{\#\ker_{\mathcal O/\mathfrak p^d}F_\g (a_1,\ldots,a_m)}
\right)^\frac12.
\]
The rest of the argument is essentially the same as the proof of Theorem~\ref{Algorithm}. We remark that in the proof given in~\cite{BMS19} we used Rado-Horn's Theorem to reduce the mimimum from being 
over $\F_p$-bases to over $\F_q$-bases. We are unable to extend this step to the setting of finite truncated valuation rings, which is why in the statement of Proposition~\ref{Algorithm-R} the minimum is taken over a family of $\F_p$-bases.
\end{proof}

\begin{definition}
Let $\{\phi_i\}_{i\in\mathcal I}$ be a finite family of polynomials
in $\Z[T_1,\ldots,T_m]$. A solution $(z_1,\ldots,z_m)\in (\mathcal O/\mathfrak p)^{\oplus m}$ to the system of equations $\phi_i=0$, $i\in\mathcal I$, is called refinable to $\mathcal O$ if there exists a solution $(\tilde z_1,\ldots,\tilde z_m)\in\mathcal O^{\oplus m}$ to this system whose reduction modulo $\mathfrak p$ equals $(z_1,\ldots,z_m)$. 
\end{definition}

In the proof of 
Theorem~\ref{thm:upperbound} 
we need 
a variation of Hensel's Lemma, which has no constraint on the derivative,  but assumes that the residual characteristic is sufficiently large.
This result 
(see Proposition~\ref{prop:filterhensel} below)
has been proved by Ax and Kochen~\cite{AxKochenI}
 over $\Q_p$ using ultrafilters, and by Greenleaf~\cite{Greenleaf} and Greenberg~\cite{Greenberg} over arbitrary $p$-adic fields (with a quantitative lower bound on $p$)
using algebraic geometry. 
An elementary argument for Greenberg's theorem was given by Kneser in~\cite{Kneser}. 
We remark that a minor adaptation of the proof of Ax and Kochen establishes Proposition~\ref{prop:filterhensel}.

  \begin{proposition}
\label{prop:filterhensel}
Let $\{\phi_i\}_{i\in \mathcal I}$ be a finite family of multivariate polynomials with integer coefficients. Then there exists a constant $M>0$
such that  for every $p$-adic field $\mathbb K$ with ring of integers $\mathcal O$, such that the characteristic of the residue field $\mathcal O/\mathfrak p$ is bigger than $M$, every solution 
in $\mathcal O/\mathfrak p$ of the system of equations $\phi_i=0$, $i\in\mathcal I$, is refinable  to a solution in $\mathcal O$.
\end{proposition}

\begin{proof}[Proof of Theorem~\ref{thm:upperbound}]
Set $\mathscr P:=\mathscr P_{i_\circ}$,  $\mathscr F:=\mathscr F_{j_\circ}$, and $g(T):=g_{i_\circ j_\circ}(T)$.
As in Corollary~\ref{cor:mixed},
we express $g(T)$ as   $g(T)=l_2+\hat{g}(T)$ where $\hat{g}(T):=\sum_{i=0}^{\deg(\hat{g})}\hat{g}_iT^i$.
First we prove the upper bound. 
Set $\tilde g_k:=\sum_{i=0}^k \hat{g}_i$ for $0\leq k\leq \deg(\hat{g})$
and  $\tilde g_{-1}:=0$. 
Recall that $n:=\rank_\Z(\g/\ZZ(\g))$ is the size of the commutator matrix $F_\g$. For $0\leq k\leq \lfloor \frac{n}{2}\rfloor$ let 
$D^{(k)}$ be the $n\times n$ diagonal matrix 
whose diagonal entries are defined by $D^{(k)}_{i,i}=1$ for $i\leq 2k$ and $D^{(k)}_{i,i}=0$ for $i>2k$. 
Consider the system of polynomial equations
with integer coefficients in variables 
$\mathsf  
x_{a,b}
^{(i)}
$,
$\mathsf  
y_{a,b}
^{(i)}
$,
$\mathsf r^{(i)}$,
$\mathsf s^{(i)}$,
and 
$\mathsf a_j^{(i)}$, for
$1\leq a,b\leq n$, $1\leq i\leq fl_1$, and $1\leq j\leq m:=\rank_\Z([\g ,\g ])$, 
that is defined by
\begin{equation}
\label{eq:system}
\begin{cases}
X^{(i)}
F(\mathsf a_1^{(i)},\ldots,\mathsf a_m^{(i)})Y^{(i)}
=D^{(k)}\quad &\text{ for } f\tilde g_{k-1}+1\leq i\leq f\tilde g_k\text{ and }
0\leq k\leq \deg(\hat{g}),\\
\mathsf r^{(i)}\det X^{(i)}=1& 
\text{ for }
1\leq i\leq fl_1,\\
\mathsf s^{(i)}\det Y^{(i)}=1& 
\text{ for }
1\leq i\leq fl_1,\\\end{cases}
\end{equation}
 where  $X^{(i)}$ and $Y^{(i)}$ are $n\times n$ matrices with entries $\mathsf x_{a,b}^{(i)}$ and 
$\mathsf y_{a,b}^{(i)}$, respectively. 
Now fix $(p,f)\in\mathscr P\times \mathscr F$,  
and let  $\mathbf a_1,\ldots,\mathbf a_{l_1}\in (\mathcal O/\mathfrak p)^{\oplus m}$
be chosen as in Corollary~\ref{cor:mixed}.
Next choose any $\F_p$-basis $\beta_1,\ldots,\beta_f$ for $\mathcal O/\mathfrak p$, and consider the vectors 
$\mathbf a'_{i,j}:=\beta_j\mathbf a_i$ for $1\leq i\leq l_1$ and $1\leq j\leq f$.
The projections of the $\mathbf a'_{i,j}$ on the first $l_1$ coordinates form a basis for the $\F_p$-vector space 
$(\mathcal O/\mathfrak p)^{\oplus l_1}$. 
Furthermore, 
\[
\#\left\{(i,j)\,:\,
\mathrm{rk}_{\mathcal O/\mathfrak p}(F(\mathbf a'_{i,j}))={2k}
\right\}=f\hat{g}_k,
\]
because
$\#\ker_{\mathcal O/\mathfrak p}(F(\beta_j\mathbf a_i))$
does not depend on $j$. It follows that for all $(p,f)\in \mathscr P\times \mathscr F$, the polynomial system~\eqref{eq:system} has a solution in $\mathcal O/\mathfrak p$, where
 the vectors $(\mathsf a_1^{(u)},\ldots,\mathsf  a_m^{(u)})$ for 
$f\tilde g_{k-1}+1\leq u\leq f\tilde g_k$ 
are equal to the  $\mathbf a'_{i,j}$ satisfying $\mathrm{rk}_{\mathcal O/\mathfrak p}(F(\mathbf a'_{i,j}))=2k$.  
By Proposition~\ref{prop:filterhensel}, for $p$ sufficiently large (depending only on the system~\eqref{eq:system}, hence only on $\g$) this solution is refinable to a solution in $\mathcal O$. In particular there exist vectors $\tilde{\mathbf a}_i\in\mathcal O^{\oplus m}$ for $1\leq i\leq fl_1$ such that $X_iF(\tilde{ \mathbf a}_i)Y_i=D^{(k)}$ for $f\tilde g_{k-1}+1\leq i\leq f\tilde g_k$ and 
$0\leq k\leq \deg(\hat{g})$, where $X_i$ and $Y_i$ are invertible $n\times n$ matrices with entries in $\mathcal O$. If $\varpi\in\mathfrak p$ is a uniformizer then for $0\leq j\leq e-1$ we have 
$
\#\ker_{\mathcal O/\mathfrak p^d}F(\varpi^j\tilde{\mathbf a}_i)=p^{fd(n-2k)+2kfj}
$. 
Since 
$\mathrm{rk}_\Z(\g) -l_1-l_2=n$, we obtain
\[
\left(
\frac{p^{fd(\mathrm{rk}_\Z(\g) -l_1-l_2)}}{\#\ker_{\mathcal O/\mathfrak p^d}(F(\varpi^j\tilde{\mathbf a}_i))}
\right)^\frac12=p^{kf(d-j)}.
\]
Furthermore, the projections to the first $l_1$ coordinates of the list of vectors \[
\{\varpi^j\tilde{\mathbf a}_i\ :\ 0\leq j\leq e-1\text{ and }1\leq i\leq fl_1\}\] 
forms a basis of the $\F_p$-vector space $(\mathcal O/\mathfrak p^e)^{\oplus l_1}$. The upper bound part of the theorem follows from Proposition~\ref{Algorithm-R}.

For the lower bound, it suffices to  assume that $p>\max\left\{l_1+l_2,l_2+\sum_{i=0}^{\deg(\hat{g})}\hat{g}_i\right\}$.  Suppose that 
the $fel_1$-tuple 
$(\mathbf a_1,\ldots,\mathbf a_{fe l_1})$
of vectors in $(\mathcal O/\mathfrak p^d)^{\oplus m}$ corresponds as in Proposition~\ref{Algorithm-R} to the 
value of $\mf(\GG_R)$. Let $
\overline{\mathsf{proj}}:(\mathcal O/\mathfrak p^d)^{\oplus m}\to (\mathcal O/\mathfrak p)^{\oplus l_1}
$ be the natural projection map defined similar to $\mathsf{proj}$ (see~\eqref{eq:dfnofprj}). The vectors 
$\overline{\mathsf{proj}}(\mathbf a_i)$ form a spanning set of the $\F_p$-vector space $(\mathcal O/\mathfrak p)^{\oplus l_1}$, hence without loss of generality we can assume that 
$\{\overline{\mathsf{proj}}(\mathbf a_i)\}_{i=1}^{fl_1}$ is an $\F_p$-basis of $(\mathcal O/\mathfrak p)^{\oplus l_1}$. 
For $d\geq 1$ and $1\leq i\leq fl_1$ set
\[
x_{i,d}:=\left(
\frac{p^{f(\mathrm{rk}_\ZZ(\g)-l_1-l_2)}}{\#\ker_{\mathcal O/\mathfrak p^d}(F(\mathbf a_i))}
\right)^\frac12.
\]
Note that $x_{i,d}^2$ is the cardinality of the subgroup of $(\mathcal O/\mathfrak p^d)^{\oplus m}$ that is generated by the columns of $F(\mathbf a_i)$, and since
$F(\mathbf a_i)$ is a skew symmetric matrix,  we also have $x_{i,d}=p^{fy_{i,d}}$ for a non-negative integer $y_{i,d}$.
 It follows that $x_{i,d}\geq x_{i,1}^d$. 
By Proposition~\ref{Algorithm-R} (for $R=\F_{p^f}$) and Theorem~\ref{thm:mixed},
\begin{equation}
\label{eq:fl2+}
fl_2+\sum_{i=1}^{fl_1}
p^{fy_{i,1}}\geq fg(p^f).
\end{equation}
Setting $y_{i,d}=0$ for $fl_1+1\leq i\leq f(l_1+l_2)$ and $d\geq 1$, from~\eqref{eq:fl2+} it follows that  
\begin{equation}
\label{eq:ffkjf}
\sum_{i=1}^{f(l_1+l_2)}
p^{fy_{i,1}}\geq fg(p^f).
\end{equation}
Our strategy is to prove that 
\begin{equation}
\label{eq:pfddd}
\sum_{i=1}^{f(l_1+l_2)}
p^{fdy_{i,1}}\geq fg(p^{fd})
.\end{equation}
Without loss of generality we 
assume that $y_{1,1}\geq \cdots\geq y_{f(l_1+l_2),1}$. Set $N:=\deg(\hat{g})$. 
There are two cases to consider:

\emph{Case 1.} $y_{i,1}\leq N$ for all $i$. We claim that  
 $y_{i,1}=N$
for $1\leq i\leq f\hat{g}_N$: otherwise  
since $p\geq 1+l_1+l_2$ and the coefficients of $\hat{g}(T)$ are non-negative, we have
\[
\sum_{i=1}^{f(l_1+l_2)}
p^{fy_{i,1}}\leq (f\hat{g}_N-1)p^{fN}+f(l_1+l_2)p^{f(N-1)}<f\hat{g}_Np^{fN}\leq f \hat{g}(p^f)\leq fg(p^f),
\]
which is a contradiction. 
It follows from above that 
$y_{i,1}=N$ for $1\leq i\leq f\hat{g}_N$. 
If $y_{f\hat{g}_N+1,1}=N$
then from $p\geq l_2+\sum_{i=0}^N\hat{g}_i$ it follows that
\begin{align*}
\mf(\GG_R)&\geq 
\sum_{i=1}^{f(l_1+l_2)}
x_{i,d}
\geq 
\sum_{i=1}^{f(l_1+l_2)}
x_{i,1}^d
=\sum_{i=1}^{f(l_1+l_2)}
p^{fdy_{i,1}}
\geq 
(f\hat{g}_N+1)p^{fdN}\\
&= f\hat{g}_Np^{fdN}+p^{fd(N-1)}p^{fd}
\geq f\hat{g}_Np^{fdN}+
p^{fd(N-1)}
f\left(l_2+\sum_{i=0}^{N-1}
\hat{g}_i\right)
\geq fg(p^{fd}).
\end{align*}
If $y_{f\hat{g}_N+1,1}<N$, then
by cancelling out the summands $p^{fN}$ from both sides of~\eqref{eq:ffkjf} we obtain a similar relation for 
a polynomial of lower degree on the right hand side of~\eqref{eq:pfddd}, and we can repeat the above argument (leading to either Case 1 above, or Case 2 below).

\emph{Case 2.} $y_{1,1}\geq N+1$. Then 
from $p\geq l_2+\sum_{i=0}^N\hat{g}_i$ it follows that
\begin{align}
\mf(\GG_R)\geq p^{fdy_{1,1}}
\geq p^{fd(N+1)}
&\geq (1+l_2+\sum_{i=0}^N\hat{g}_i)^fp^{fdN}
\\
\geq fl_2+f
p^{fdN}\sum_{i=0}^N\hat{g}_i\geq fg(p^{fd}).
&\qedhere
\end{align}
\end{proof}

\section{Polynomiality over rings for pattern groups}

In this section we prove Theorem~\ref{partial-order-Lie}, which in particular establishes Conjecture~\ref{conjj} for pattern groups. Throughout this section we set  $\g:=\g_\prec$.  We begin with recalling some notation and general facts from~\cite{BMS19} and~\cite{BMS}. 
For an abelian $p$-group $\Gamma$,  set  
\[
\Omega_1(\Gamma):=\{g\in \Gamma: g^p=\1\}.
\]
Note that $\Omega_1(\Gamma)$ is an $\F_p$-vector space. For an abelian group $A$ we set
$\widehat{A}:=\Hom(A, \C^{\ast})$. When $A$ is an elementary abelian $p$-group,  $\widehat{A}$ is canonically an $\F_p$-vector space and there exists an isomorphism $\widehat{A}\cong\Hom(A,\FZ)$ obtained by identifying $\FZ$ with the subgroup of $p$-th roots of unity in $\C^\ast$. 

Recall that $R\cong \mathcal O/\mathfrak p^d$ for $d\geq 1$. 
Fix a primitive character $\psi: R\to \C^*$ (for example the character described in the proof of Proposition~\ref{Algorithm-R}), and also  set 
$\psi_b(x):=\psi(bx)$ for 
 $b\in R$
and $x\in R$,
so that the map
\begin{equation}
\label{eq:RRhat}
R\to \widehat{R}\ ,\ b\mapsto\psi_b
\end{equation}
is a group isomorphism.
Now let $\frak{a}$ be an ideal in $R$ and set  $\Ann(\frak{a}):=\{r\in R: r\frak{a}=0\}$. From surjectivity of the  restriction $\widehat{R}\to \widehat{\frak{a}}$ it follows that the map
\[
R/\Ann(\frak{a})\to \widehat{\frak{a}}\ , \ b+\Ann(\frak{a})\mapsto {\psi_b}\big|_{\frak{a}}
\]
%
is a group isomorphism.
Notice that $\Ann(\mathfrak{p}^k/\mathfrak{p}^d)=\mathfrak{p}^{(d-k)}/\mathfrak{p}^d$ for $0\leq k\leq d$, so that we obtain the following lemma.

\begin{lemma}\label{Character-local} Every additive character of the  group $\mathfrak{p}^k/\mathfrak{p}^d$, where
$0\leq k\leq d$, is of the form
\begin{equation*}
\psi_{b}: \mathfrak{p}^k/\mathfrak{p}^d\to \C^* \ ,
\ x\mapsto \psi(bx),
\end{equation*}
for a unique 
$b+\mathfrak{p}^{(d-k)}/\mathfrak{p}^d
\in 
(\OO{d})/(\mathfrak{p}^{(d-k)}/\mathfrak{p}^d)\cong \OO{d-k}$. In particular $\widehat{\mathfrak{p}^k/\mathfrak{p}^d}\cong \OO{(d-k)}$.
\end{lemma}


Hereafter, for a finite $p$-group $\G$ the $\FZ$-vector space $\mathrm{Hom}(\Omega_1(\ZZ(\G)),\C^*)$ will be denoted by $\widehat{\Omega}_1(\ZZ(\G))$.
By the central character of an irreducible representation $(\rho,V)$ of $\G$ we mean the group homomorphism $\chi:\ZZ(\G)\to\C^*$ satisfying $\rho(g)=\chi(g)1_V$ for $g\in\ZZ(\G)$. 
\begin{lemma}\label{central-span-faith} Let $\G$ be a finite $p$-group.
\begin{itemize}
\item[\rm (i)]
Let $(\rho_i,V_i)_{1\leq i\leq k}$ be a family of irreducible representations of $\G$ with  central characters $\chi_i$. Suppose  that $\{\restr{{\chi_i}}{\Omega_1(\ZZ(\G))}: 1\leq i\leq k\}$ spans $\widehat{\Omega}_1(\ZZ(\G))$. Then $\oplus_{1\leq i\leq k}\rho_i$ is a faithful representation of $\G$.  

\item[\rm (ii)]
Let $\rho$ be a faithful representation of $\G$ of dimension $\mf(\G)$. Then $\rho$ decomposes as a direct sum of exactly $r:=r_G$ irreducible representations,
where $r_G$ denotes the minimum number of generators of the abelian group $\ZZ(\G)$. Furthermore, the restrictions to $\Omega_1(\ZZ(\G))$ of the central characters of these representations form a $\FZ$-basis for $\widehat{\Omega}_1(\ZZ(\G))$.
\end{itemize}
\end{lemma}
\begin{proof}
For (i), see~\cite[Lem. 3.4]{BMS19} or \cite{ReichsteinI}. For (ii), see~\cite[Lem. 3.5]{BMS19}
\end{proof}


\begin{remark}
\label{rmk:mingen}
For a finite abelian group $A$, the minimum number of generators is equal to the number of invariant factors of $A$ (this is a consequence of elementary divisor theorem for finite abelian groups). By a similar argument, if $A$ is a finite abelian $p$-group, then the minimum number of generators of $A$ is equal to $\dim_{\FZ}(A\otimes_\Z \FZ)$. In particular, the minimum number of generators of the abelian group $\OO{d}$ is equal to $fe$. 
This is because $\OO{d}\otimes_\mathbb Z\FZ\cong 
(\OO{d})/(p(\mathcal O/\mathfrak p^d))
\cong 
\OO{e}$ and
$\#\mathfrak p^i/\mathfrak p^{i+1}=p^f$ for $0\leq i\leq e-1$.
\end{remark}

Next we establish a general combinatorial lemma about bases of direct sums. 

\begin{lemma}\label{Genral-decomposition} Let  $V_1$ and $V_2$ be two vector spaces of dimensions $d_1$ and $d_2$ over an arbitrary field. Let $S=\{(v_i,v'_i): 1\leq i\leq d_1+d_2\}$ be a basis for $V_1\oplus V_2$. Then $S$ can be partitioned into two sets $B_1$ and $B_2$, of sizes $d_1$ and $d_2$ respectively, such that the $V_i$-components of the vectors in $B_i$ form  a basis of $V_i$ for $i=1,2$.
\end{lemma}
\begin{proof} The following argument is communicated to us by I. Bogdanov. Let $\{e_1,\dots,e_{d_1}\}$ and $\{e'_1,\dots,e'_{d_2}\}$ denote bases of $V_1$ and $V_2$, respectively. Then for $1\leq j\leq d_1+d_2$ we obtain
\begin{equation*}
v_j =\sum_{i=1}^{d_1} a_{ij}e_i\quad\text{and}\quad
v'_j=\sum_{i=1}^{d_2} b_{ij}e'_j.
\end{equation*} 
Consider the square matrix $A$ of size $d_1+d_2$ whose $i$-th column, for $1\leq i\leq d_1+d_2$, is  
the transpose of the row vector
\[
(a_{1i},\cdots ,a_{d_1i},b_{1i},\ldots,b_{d_2i}).
\]
Since $S$ is a basis, $A$ is invertible.
Set $I:=\{1,\dots,d_1\}$ and consider the generalized Laplace expansion
$$
\det(A)=\sum_{\substack{J\subseteq \{1,\dots,d_1+d_2\}\\ |J|=d_1}} (-1)^{\sum_{i\in I} i+\sum_{j\in J} j} \det A_{I,J}\det A_{I',J'},
$$
where $I':=\{d_1+1,\ldots,d_1+d_2\}$,  $J'$ is the complement of $J$ in $\{1,\ldots,d_1+d_2\}$, and $A_{I,J}$ (respectively, $A_{I',J'}$) is the 
minor of $A$ corresponding to rows and columns indexed 
by $I$ and $J$ (respectively, by $I'$ and $J'$). Since $\det(A)\neq 0$, there exists $J$ with $\det A_{I,J}\det A_{I',J'}\neq 0$. 
The desired sets are $B_1:=\{(v_j,v'_j): j\in J\}$ and $B_2:=\{(v_j,v'_j): j\in J'\}$. 
\end{proof}
Lemma~\ref{Genral-decomposition} and induction yield the following.
\begin{lemma}\label{Lablace-identity} Let $V$ be a finite dimensional vector space over an arbitrary field and 
let 
$S$ be a basis for $V^{\oplus k}$ where $k\geq 1$. Then $S$ can be partitioned into subsets $B_1,\dots, B_k$ such that 
each of the sets $
\{\pi_\ell(w): w\in B_\ell\}
$
is a basis of $V$
for $1\leq \ell\leq k$, where
$\pi_{\ell}(v_1,\dots,v_k):=v_\ell$. 
\end{lemma}

The following simple lemma was also used in~\cite{BMS}.
 \begin{lemma}\label{ineq}
Let $a_0, \dots,  a_{ m -1}$ be non-negative real numbers with $\sum_{\ell=0}^{m-1}a_\ell=m$. Assume that for any $ 0  \le \ell \le m-1$ we have $ a_\ell + \cdots + a_{m-1} \le m-\ell$. Then for any decreasing sequence $x_0 \ge \cdots \ge x_{m-1}$, we have 
$ \sum_{\ell=0}^{m-1 } a_\ell x_\ell \ge \sum_{\ell=0}^{m-1} x_\ell$.
\end{lemma}

\begin{proof}  Define $A_\ell:=\sum_{j=\ell}^{m-1}a_j$ 
for  $0 \leq \ell\leq m-1$
and set $A_m:=0$. Then $A_\ell\leq m-\ell$ and $A_0=m$. Set $\Delta A_\ell:=A_\ell-A_{\ell+1}$ and $\Delta x_\ell:=x_\ell-x_{\ell-1}$. Since $\Delta x_\ell\leq 0$, we have 
\[
\sum_{\ell=0}^{m-1 } a_\ell x_\ell =\sum_{\ell=0}^{m-1}\Delta A_\ell x_\ell
=A_0x_0+\sum_{\ell=1}^{m-1}A_\ell\Delta x_\ell
 \ge m x_{0}+ \sum_{\ell=1}^{m-1} (m-\ell)\Delta x_\ell = \sum_{\ell=0}^{m-1} x_\ell.
\qedhere
\]
\end{proof}

For the proof of Theorem~\ref{partial-order-Lie} we need some standard facts from the \emph{orbit method} for describing irreducible representations of finite $p$-groups. The tools that we need from the orbit method can be found in various references, including~\cite{Kazhdan},~\cite{HoweII}, and~\cite{Boyarchenko}.  For the reader's convenience, we briefly review the pertinent results from the orbit method (we follow~\cite{Boyarchenko} closely; see also~\cite[Sec. 3.2]{BMS19}). 
Let $\mathfrak f$ and $\exp(\mathfrak f)$ be as in Section~\ref{sec:intro}. Then we can define a coadjoint action of $\exp(\mathfrak f)$ on
$\widehat{\mathfrak f}:=\Hom_\Z(\mathfrak f,\C^*)$
as follows. For $x\in\exp(\mathfrak f)=\mathfrak f$ and $\theta\in \widehat{\mathfrak f}$, we define $\theta^x\in\widehat{\mathfrak f}$ by 
\[\theta^x(y):=\theta\left(\sum_{k\geq 0} \frac{\ad_x^k(y)}{k!}\right),\quad\text{ for } y\in \mathfrak f. 
\]
Note that only finitely many summands are nonzero. 
\begin{theorem}\label{Kirillov}
Assume that $p$ is an odd prime that is strictly larger than the nilpotency class of $\mathfrak f$. 
Then there exists a bijection between the orbits of the coadjoint action on $ \widehat{\mathfrak f}$ and characters of irreducible representations of $\exp(\mathfrak f)$.
Furthermore, if $\rho_\Theta$ denotes the irreducible representation corresponding to the coadjoint orbit $\Theta\subseteq\widehat{\mathfrak f}$, then the following statements hold. 
\begin{itemize}
\item[\rm (a)]
The character of $\rho_\Theta$ is given by
$
\displaystyle\chi_\Theta(x):=|\Theta|^{-1/2}\sum_{\theta\in\Theta}\theta(x)
$ for $x\in \exp(\mathfrak f)=\mathfrak f$.
\item[\rm (b)]
$\displaystyle \dim(\rho_\Theta)=
\left(\frac{\#\mathfrak f}{\# \mathrm{Stab}_\mathfrak f(\theta_0)}\right)^{1/2}
$
for $\theta_0\in\Theta$, where
$
\mathrm{Stab}_\mathfrak f(\theta_0):=
\left\{
x\in\mathfrak f\,:\,
\theta_0([x,\mathfrak f])=1\right\}
.
$
\item[(c)] $\rho_\Theta(g)$ is multiplication by the scalar $\theta(g)$ for $g\in \ZZ(\exp(\mathfrak f))$ and $\theta\in\Theta$. 
\end{itemize}

\end{theorem}

Recall that in this section $\g:=\g_\prec$. Our next goal is to use the orbit method to study irreducible representations of $\GG_R:=\exp(\g_R)$. 
\begin{remark}
\label{rmk:5.2}
The nilpotency class of 
$\g$ is 
$\max\{\alpha(i,j): i\prec j\}+1$, where $\alpha(i,j)$ is defined as in~\eqref{eq:alphai,j}. To see this, note that the $l$-th term in the lower central series of $\g$ is the $\Z$-span of the $e_{ij}$ where $\alpha(i,j)\geq l$. 
It is also straightforward to verify that $\ZZ(\g)$ is the $\Z$-span of the $e_{ij}$ where $(i,j)\in \Iex$ 
(see~\cite[Lemma 4.2]{BMS19}). The same argument proves that $\ZZ(\g_R)$ is a free $R$-module generated by the $e_{ij}$ for $(i,j)\in\Iex$.  
\end{remark}

By the isomorphism~\eqref{eq:RRhat}, elements of
 $\widehat{\g_R}$ are of the form $\psi_\bb$ for a vector $\bb:=(b_{ij})_{i\prec j}$ with entries in $R$, where
\begin{equation}
\psi_{\bb}\left(\sum_{i\prec j} x_{ij}e_{ij}\right):=\psi\left(\sum_{i\prec j}b_{ij}x_{ij}\right).
\end{equation}
\begin{definition}
\label{dfn:level}
We define the level of  $b:=\tilde{b}+\mathfrak{p}^d\in R$ to be the smallest $k$ such that $\tilde{b}\not\in \mathfrak{p}^{k+1}$, and we denote it by $\lev(b)$.
\end{definition}

\begin{proposition}\label{level-prop-Min} Let $b_{ij}=\tilde{b}_{ij}+\mathfrak{p}^d$ for $i\prec j$ be  elements of $R$. Set $\bb:=(b_{ij})_{i\prec j}$, and let $\rho_{_\bb}$ be the irreducible representation of $\GG_R$ corresponding to the orbit of $\psi_\bb$ under the coadjoint action.  Then 
$
\dim \rho_{_\bb}\geq p^{f(d-\lev(b_{i_1j_1}))\alpha(i_1,j_1)}
$
for all $(i_1,j_1)\in\Iex$. 
\end{proposition}

\begin{proof} 
We use Theorem~\ref{Kirillov}(b). 
Let $x=\sum_{i\prec j}x_{ij}e_{ij}\in \g_R$ be an element of $\mathrm{Stab}_{\g_R}(\psi_{\bb})$. Then for any $y=\sum_{i\prec j}y_{ij}e_{ij}\in \g_R$ we have   
\begin{equation*}
\begin{split}
1=\psi_{\bb}([x,y])&=\psi_{\bb}\left(\sum_{i\prec j,\, k\prec l} x_{ij}y_{kl}[e_{ij},e_{kl}]\right)\\
&=\psi\left(\sum_{i\prec j}\sum_{i\prec k\prec j}b_{ij}(x_{ik}y_{kj}-x_{kj}y_{ik})\right)
=\psi\left(\sum_{i\prec j}\left(\sum_{k\prec i}b_{kj}x_{ki}-\sum_{j\prec l}b_{il}x_{jl}\right)y_{ij}\right).
\end{split}
\end{equation*}
Since the $y_{ij}\in R$ are arbitrary and $\psi$ is primitive, the stabilizer of 
$\psi_{_\bb}$ is equal to the solution set of the system of linear equations 
$\mathscr L:=\left\{L_{ij}(x_{st})=0\ :\ i,j\in[n]\text{ and }i\prec j\right\}$, where
\begin{equation}\label{Linear-eq1}
L_{ij}(x_{st}):=\sum_{k\prec i}b_{kj}x_{ki}-\sum_{j\prec l}b_{il}x_{jl}.
\end{equation}
In particular since $(i_1,j_1)\in\Iex$, the linear forms $L_{i_1i}(x_{st})$ and $L_{jj_1}(x_{st})$, for $i_1\prec i\prec j_1$ and $i_1\prec j\prec j_1$, yield   $2\alpha(i_1,j_1)$ linear equations that can be written as
\begin{equation}\label{Linear-eq2}
b_{i_1j_1}x_{ij_1}=-\sum_{i\prec k\neq j_1}b_{i_1k}x_{ik}\quad\text{ and }\quad
b_{i_1j_1}x_{i_1j}=-\sum_{i_1\neq l\prec j}b_{lj_1}x_{lj}.
\end{equation}

The cardinality  of the kernel of the map $R\to R,\, x\mapsto b_{i_1j_1}x$, is $p^{f\lev(b_{i_1j_1})}$. 
To see this, note that the ideal generated by $b_{i_1j_1}$ is $\mathfrak{p}^{\ell}/\mathfrak{p}^d$, with $\ell=\lev(b_{i_1j_1})$. The 
above fact now follows from $\Ann(\mathfrak{p}^\ell/\mathfrak{p}^d)=\mathfrak{p}^{(d-\ell)}/\mathfrak{p}^d$ in combination with $\# (\mathfrak{p}^i/\mathfrak{p}^{i+1})=p^f$ for $0 \le i \le d-1$.

It follows that for any choice of values for the variables on the right hand side of~\eqref{Linear-eq2}, there exists at most $p^{f\lev(b_{i_1j_1})}$ choices for each of the $2\alpha(i_1,j_1)$ variables 
$x_{i_1i}$ and $x_{jj_1}$ such that the corresponding equation in~\eqref{Linear-eq2} is satisfied.  
Therefore the number of solutions of the  linear system $\mathscr L$ is  
at most 
\begin{equation}\label{Stab}
p^{fd(|I|-2\alpha(i_1,j_1))+2f\alpha(i_1,j_1)
\lev(b_{i_1j_1})},
\end{equation}
where  $I:=\{(i,j): i\prec j\}$.
Using Theorem~\ref{Kirillov}(b) it is now straightforward to verify that 
$
\dim \rho_{_\bb}\geq p^{f(d-\lev(b_{i_1j_1}))\alpha(i_1,j_1)}
$.
\end{proof}

\begin{lemma}\label{level-lemma} Let $b=\tilde{b}+\mathfrak{p}^d\in R$. Fix $(i_1,j_1)\in \Iex$ and let ${\bf b}:=(b_{ij})_{i\prec j}$ where $b_{i_1j_1}=b$ and $b_{ij}=0$ for all other pairs $(i,j)$. 
 Then the dimension of the irreducible representation of $\GG_R$ that corresponds to the coadjoint orbit of $\psi_{\bf b}\in\widehat{\g_R}$  is equal to $q^{(n-\lev(b)) \alpha(i_1,j_1)}$.
\end{lemma}

\begin{proof}
We write the linear system in the proof of Proposition~\ref{level-prop-Min} explicitly, and determine the number of its solutions. In this case the only nonzero equations in the linear system $\mathscr L$ are the $L_{ij_1}$ and the $L_{i_1j}$ for $i_1\prec i,j\prec j_1$, and they are  of the form $b_{i_1j_1}x_{i_1j}=0$ and $b_{i_1j_1}x_{ij_1}=0$. 
The values of the variables $x_{ij_1}$ and $x_{i_1j}$ can be chosen independently of each other, and the only restriction is that \[
\lev(x_{i_1j}),\lev(x_{ij_1})\ge d-
\lev(b_{i_1j_1})
.
\]
Thus, for each of these variables there are $p^{f\lev(b_{i_1j_1})}$ possible values. 
There is no restriction on the values of the other variables of the linear system and each of them can be chosen arbitrarily, from $p^{fd}$ possible values.
Thus the number of solutions of the linear system is equal to~\eqref{Stab}. 
Since the number of solutions of the linear system is also equal to the cardinality of $\mathrm{Stab}_{\g_R}(\psi_{\bb})$, 
the assertion of the lemma follows from Theorem~\ref{Kirillov}(b).  
\end{proof}

In the next lemma we prove that
$
\mf(\GG_R)\leq \sum_{\ell=0}^{ e-1 }\sum_{(i,j)\in \Iex} fp^{f(d-\ell)\alpha(i,j)}
$.

\begin{lemma}\label{cons-faith} The group $\GG_R$ has a faithful representation of dimension 
\[
\sum_{\ell=0}^{e -1}\sum_{(i,j)\in \Iex} f p^{f(d-\ell)\alpha(i,j)}
.\]
\end{lemma}

\begin{proof}
Let $\omega_1,\dots,\omega_f$ be units in $\mathcal{O}$ such that $\{\omega_1+\mathfrak{p},\dots,\omega_f+\mathfrak{p}\}$ forms a basis for $\mathcal{O}/\mathfrak{p}$ over $\F_p$. Set  
$\tilde{b}_{k\ell}:=\omega_k\varpi^\ell$ 
and $b_{k\ell}:=\tilde{b}_{k\ell}+\mathfrak p^e$ 
for $1\leq k\leq f$ and $0\leq \ell\leq e-1$.
The ${b}_{k\ell}$ 
form a basis of the $\F_p$-vector space $\OO{e}$.
Since $\ZZ(\GG_R)\cong\ZZ(\g_\R)$ as abelian groups, 
from Remark~\ref{rmk:5.2} it follows that
\begin{equation}\label{Z(G)}
\widehat{\Omega}_1(\ZZ(\GG_R))\cong\widehat{\Omega}_1(\ZZ(\g_\R))=\bigoplus_{(i,j)\in \Iex} \widehat{\Omega}_1(\OO{d}).
\end{equation}
Note that $\Omega_1(\OO{d})=\mathfrak{p}^{d-e }/\mathfrak{p}^{d}$, so that Lemma~\ref{Character-local} yields an isomorphism of abelian groups 
$
\OO{e}
\cong
\widehat{\Omega}_1(\OO{d}).
$
Given $i_\circ,j_\circ\in [n]$ such that  $i_\circ\prec j_\circ$, we define 
${\bf b}({i_\circ,j_\circ,k,\ell})$, 
 for
  $1\leq k\leq f$ and $0\leq \ell\leq e-1$,
to be the vector in 
 $\bigoplus_{i\prec j}\OO{d}$ with exactly one nonzero component, at the $(i_\circ,j_\circ)$-position, equal to $\tilde b_{k\ell}$.

Since $\ZZ(\g_R)\cong\bigoplus_{(i,j)\in\Iex}\OO{d}$ (see Remark~\ref{rmk:5.2}), from~\eqref{Z(G)} it follows that 
the restrictions of the characters $\psi_{\mathbf b(i_\circ,j_\circ,k,l)}$
to $\ZZ(\GG_R)$ form an $\F_p$-basis 
of $\widehat{\Omega}_1(\ZZ(\GG_R))$. Let $\rho_{\mathbf b(i_\circ,j_\circ,k,\ell)}$ denote the irreducible representation of $\GG_R$ that corresponds to the coadjoint orbit of $\psi_{\mathbf b(i_\circ,j_\circ,k,\ell)}$ (see Theorem~\ref{Kirillov}).
Then by Lemma~\ref{central-span-faith}, the representation
$$\rho:=\bigoplus_{\substack{1\leq k\leq f\\ 0\leq \ell\leq e-1}}\bigoplus_{(i,j)\in \Iex}\rho_{{\bf b}(i,j,k,\ell)}$$
is faithful. The set $\{b_{k\ell}\,:\,1\leq k\leq f,\ 0\leq \ell\leq e-1\}$
contains exactly $f$ elements of any given level $\ell$ where $0\leq \ell\leq e-1$, and
by Lemma~\ref{level-lemma} we have $\dim\rho_{\mathbf b(i,j,k,\ell)}=p^{f(d-\ell)\alpha(i,j)}$ for any $(i,j)\in\Iex$. These imply that
$
\dim\rho=\sum_{\ell=0}^{e -1}\sum_{(i,j)\in \Iex}fq^{(n-\ell)\alpha(i,j)}.$
\end{proof}

\begin{proof}[Proof of the Theorem~\ref{partial-order-Lie}]
By Lemma~\ref{cons-faith} it suffices to prove that
\[
\mf(\GG_R)\geq \sum_{\ell=0}^{ e-1 }\sum_{(i,j)\in \Iex} fp^{f(d-\ell)\alpha(i,j)}
.
\]
Let $\rho$ be a faithful representation of $\GG_R$ of dimension $\mf(\GG_R)$.
By Lemma~\ref{central-span-faith}(ii) and 
Remark~\ref{rmk:mingen}, 
$\rho$ is a direct sum of  $N:=ef(\#\Iex)$ irreducible representations. Thus we can express $\rho$ as
\[\rho=\bigoplus_{k=1}^N\rho_{\psi_{\mathbf a_k}},
\]
where each $\rho_{\psi_{\mathbf a_{k}}}$ is the irreducible representation of $\GG_R$ 
corresponding (according to Theorem~\ref{Kirillov}) to the coadjoint orbit of the character 
$\psi_{\mathbf a_k}$ of $\g_R$ (and $\mathbf a_k$ is a vector in $\bigoplus_{i\prec j}\OO{d}$). 
Furthermore, by Theorem~\ref{Kirillov}(c) and Lemma~\ref{central-span-faith}(ii) the restrictions of the $\psi_{\mathbf a_k}$ to $\Omega_1(\ZZ(\GG_R))$ form a $\FZ$-basis of $\widehat{\Omega}_1(\ZZ(\GG_R))$. Since
$
\Omega_1(\OO{d})=\mathfrak p^{d-e}/\mathfrak p^d
$, Lemma~\ref{Character-local} implies that 
$\widehat{\Omega}_1(\OO{d})\cong \OO{e}$ as $\FZ$-vector spaces, from which it follows that the projections of the vectors $\mathbf a_k$ onto 
$\bigoplus_{(i,j)\in\Iex}\OO{e}$ form a $\FZ$-basis. 
Using Lemma~\ref{Lablace-identity} we can partition the set $\{\mathbf a_k\}_{k=1}^N$ into $\#\Iex$ sets of cadinality  $ef$, say $\{\mathscr A_{(r,s)}\,:\,(r,s)\in\Iex\}$,  such that the $(r,s)$-components of  the elements of $\mathscr A_{(r,s)}$  form a basis of $\OO{e}$. 
To complete the proof, it suffices to verify that 
\begin{equation}
\label{eq:sumaArs}
\sum_{\mathbf a\in \mathscr A_{(r,s)}}
\dim(\rho_{\psi_\mathbf a})\geq f\sum_{\ell=0}^{e-1}
p^{f(d-\ell)\alpha(r,s)}\quad\text{ for }(r,s)\in\Iex.
\end{equation}
Fix $(r,s)\in\Iex$, and  denote the $(r,s)$-component of any  $\mathbf a\in \mathscr A_{(r,s)}$ by $\mathbf a_{(r,s)}$. For $0\leq \ell\leq e-1$ set 
\[
N_\ell:=\#\{\mathbf a\in\mathscr A_{(r,s)}\,:\,\lev(\mathbf a_{(r,s)})=\ell\},
\]
where $\lev(\cdot)$ is the level as in Definition~\ref{dfn:level}.
From Proposition~\ref{level-prop-Min} it follows that 
if $\lev(\mathbf a_{(r,s)})=\ell$ then 
$\dim (\rho_{\psi_{\mathbf a}})\geq p^{f(d-\ell)\alpha(r,s)}$. This implies that 
\begin{equation}
\label{eq.I}
\sum_{\mathbf a\in\mathscr A_{(r,s)}}\dim(\rho_{\psi_{\mathbf a}})\geq \sum_{\ell=0}^{e-1}N_\ell
p^{f(d-\ell)\alpha(r,s)}.
\end{equation}
Note that $\sum_{\ell=0}^{e-1}N_\ell=ef$, and
 \[\sum_{k=\ell}^{e-1}N_k\leq \dim_{\FZ}(\mathfrak p^\ell/\mathfrak p^e)=(e-\ell)f.
 \] Thus by Lemma~\ref{ineq}
 for $m:=e$, $a_\ell:=N_\ell/f$ and $x_\ell:=p^{f(d-\ell)\alpha(r,s)}$ we obtain 
 \begin{equation}
 \label{eq:II}
 \sum_{\ell=0}^{e-1}N_\ell
p^{f(d-\ell)\alpha(r,s)}\geq f\sum_{\ell=0}^{e-1}
p^{f(d-\ell)\alpha(r,s)}
 .\end{equation}
Inequality~\eqref{eq:sumaArs} now follows from~\eqref{eq.I} and~\eqref{eq:II}.
\end{proof}

\section{The faithful dimension in the case $\g:=\m_{n,c}$}

In this section we prove  
Theorem~\ref{meta}. We begin  by introducing some notation.
Recall that $[n]:=\{1,\ldots,n\}$ for $n\in\mathbb N$.  
 For $k \ge 1$, the set of sequences $$\i=(i_1, \dots, i_k), 
\quad i_1, \dots, i_k \in [n]
$$ is denoted by $\A(n, k)$. 
We say that $\i$ is \emph{decreasing} if  $ i_1 \ge  \dots \ge i_k$. The set of decreasing sequences
in $\A(n, k)$ is denoted by $\D(n,k)$.
There is an obvious sorting map
$$\bar{ \, }: \A(n,k) \to \D(n,k).$$  For example $ \overline{(5,4,4,2,3,1)}=(5,4,4,3,2,1)$. 
We say that $\i \in \A(n,k)$ is a 
\emph{Hall sequence} if  the initial sequence $(i_1, \dots, i_{k-1})$ is decreasing and $i_{k-1} < i_{k}$ (by convention, every element of $\A(n,1)$ is a Hall sequence). The 
subset of Hall sequences in $\A(n,k)$ is denoted by $\Hall(n, k)$. 
\begin{lemma}
\begin{enumerate}
\item[(a)] For $n, k \ge 1$ we have
$\displaystyle \# \D(n, k)= {n+k-1 \choose k}.$
\item[(b)] For $n \ge 2$ and $ k \ge 2$ we have 
\[ \# \Hall(n, k)= \sum_{ m =1}^{n-1} m {k+m-2 \choose k-2}= (k-1) {k+n-2 \choose k}.  \]
\end{enumerate}
\end{lemma}

\begin{proof} 
For part (a), note that cardinality of $\D(n,k)$ is equal to the number of solutions of the equation $t_1+ \cdots + t_n=k$ in non-negative integers, where  $t_j = \# \{ \ell \in [k]: i_\ell= j \}$. 

For part (b), fix $n, k \ge 2$ and for $m \in [n-1]$ denote by $A_m$ the subset of $ \Hall(n,k)$ consisting of those sequences
$(i_1, \dots, i_k) \in \Hall(n,k)$ such that $ i_{k-1}=m$. Note that since $i_{k-1}< i_k \le n$, we must have $i_{k-1} \in [n-1]$. Once $m$ is fixed, there are $ n-m$ options left for $i_k$. Moreover, the sequence $(i_1, \dots, i_{k-1})$ can be any decreasing sequence in $[n]$ for which $i_{k-1} = m$. Using the map $(i_1, \dots, i_{k-2}) \mapsto (i_1-m+1, \dots, i_{k-2}-m+1)$ these sequences stand in one-to-one correspondence with elements of the set $\D(n-m+1, k-2)$, whose cardinality by part (a) is
equal to $ {n-m+k-2 \choose k-2}$. Since the two  choices can be made independently of each other 
we have 
\[ \# \Hall(n, k)= \sum_{ m =1}^{n-1} (n-m) {n-m+k-2 \choose k-2}. \]
Replacing $m$ by $n-m$ in this expression yields the first equality in part (b). 
For the second equality note that
\begin{align*}
\sum_{m=1}^{n-1}
m{k+m-2\choose k-2}=
\sum_{m=1}^{n-1}
(k-1){k+m-2\choose k-1}=
(k-1)
\sum_{m=1}^{n-1}
{k+m-2\choose k-1},
\end{align*}
and $\sum_{m=1}^{n-1}
{k+m-2\choose k-1}$  counts the number of subsets  of size $k$ of $[k+n-2]$, where the summand ${k+m-2\choose k-1}$ corresponds to subsets 
of $[k+n-2]$
with maximum equal to $k+m-1$. 
\end{proof}

Since 
$\f_{n,2}=\m_{n,2}$,
the case $c=2$ of Theorem~\ref{meta}
follows from \cite[Theorem 2.13]{BMS}. Thus  from now on we assume that $c \ge 3$. 
We denote the standard generators of $\m_{n,c}$ by $X_n := \{ \sv_i: i \in [n] \}$.
For each $k$-tuple $\i= (i_1, \dots, i_k) \in \A(n,k)$, write
\[ \sv_\i:= 
[\sv_{i_1},[\sv_{i_{2}}, \dots, [\sv_{i_{k-1}}, \sv_{i_k}] \dots ]. \]

 Set
$\Hall^{\le c-1}:= \bigcup_{ j=1}^{  c-1}  \Hall(n, j)$ and 
$ \Hall^{\ge 2}:= \bigcup_{  j=2}^{  c}  \Hall(n, j)$.

\begin{lemma}\label{bases}
 Let $n$ and $c$ be as above. For each $ 1 \le j \le c$, the set $\{ \sv_\i: \i \in \Hall(n,j) \}$ is a basis for 
the vector space $\m_{n,c}^j/\m_{n,c}^{j+1}$. In particular, 
\begin{enumerate}
\item[(a)]  $\{ \sv_\i: \i \in \Hall(n,c) \}$ forms a basis for $\ZZ(\m_{n,c})$.  
\item[(b)] Elements in 
$\left\{\sv_\i\,:\,\i\in\Hall^{\le c-1}\right\}$ represent a basis of $\m_{n,c}/\ZZ(\m_{n,c})$. 
\item[(c)] Elements in $\left\{\sv_\i\,:\,\i\in\Hall^{\ge 2}\right\}$ form a basis for $ [\m_{n,c}, \m_{n,c}]$. 
\end{enumerate}

\end{lemma}
\begin{proof}
This is well known. See~\cite{ABR} for a proof.
\end{proof}

In the sequel, we will carry out a close study of the commutator matrix of $\m_{n,c}$ with respect to the bases for $\m_{n,c}/\ZZ(\m_{n,c})$ and
$ [\m_{n,c}, \m_{n,c}]$ described in~Lemma \ref{bases}. 
For $ \i, \j \in \Hall^{\le c-1}$ 
let $\lambda_{\i\j}^\k\in\Z$ be the structural constants defined by 
\begin{equation*}
[\sv_\i, \sv_\j]= \sum_{\k \in \Hall^{\ge 2}} \lambda^{\k}_{\i \j} \sv_\k.
\end{equation*}
Let $\T:=(T_\i)_{\i\in \Hall^{\geq 2}}$ 
be a vector of variables, and for $\i, \j \in \Hall^{\le c-1}$, define
the linear forms
\begin{equation*}
\Lambda_{\i \j}(\T):= \sum_{\k \in \Hall^{\ge 2}} \lambda_{\i \j}^\k T_\k.
\end{equation*}
 The commutator matrix of $\g$
relative to the chosen bases is the $N\times N$ skew-symmetric matrix of linear forms given by 
\[ F_\g(\T):=[ \Lambda_{\i \j}(\T)]_{\i, \j \in \Hall^{ \le c-1} },
\quad\text{where  $N:=\#\Hall^{\le c-1}$.}
 \]

 For any 
$(i_1,\ldots,i_c)\in \Hall(n,c)$, set $\i:=(i_1)$ and $\j:=(i_2,\ldots,i_c)\in\Hall(n,c-1)$. Then the $(\i, \j)$ entry of  
$ F_\g(\T)$ is equal to $T_{(i_1,\ldots,i_c)}$.  
Further, for any $\i :=(i_1, \dots, i_k)$ and $\j := (j_1, \dots, j_l)$, we have $ \Lambda_{\i \j}=0$ unless $\min(k, l)=1$. 

%
%

\begin{lemma}\label{comcal}
Let  $\i:=( i_1)\in \Hall(n,1)$ and $ \j:= (i_2, \dots, i_c) \in \Hall(n, c-1)$. Then
\[
{ \Lambda}_{\i \j}( \T)=
\begin{cases}
T_{(\overline{i_1, \dots, i_{c-1}}, i_c)}& \text{ if }\
i_1 \ge i_{c-1},\\
T_{(i_2, \dots, i_{c-1}, i_1, i_c)}- T_{( \overline{i_2, \dots, i_{c-2}, i_c}, i_1, i_{c-1}) }
&\text{ if }\
i_1<i_{c-1}.
\end{cases}
\]

\end{lemma}

\begin{proof}
First observe that 
for $a, a' \in  \m_{n, c}$ and  $ a'' \in [ \m_{n, c},  \m_{n, c}]$ we have
\begin{equation}
\label{Jac} 
[a, [a',a''] ]= [a', [a, a''] ],
\end{equation}
because  
$[ a, [a', a''] ]= -[a', [a'', a] ]- [ a'', [a, a'] ]=  [a', [a, a''] ],$
where the second equality follows from the fact that $ \m_{n, c}$ is metabelian and $a''$ is a commutator.

Suppose that $ i_1 \ge i_{c-1}$,
and let $2 \le r \le c-1$ be the smallest integer with $ i_1 \ge i_r$. By the assumption such an $r$ exists. If $c =3$, then
 we must have $r=2$, hence 
 $i_1 \ge i_2$. Thus 
 $(i_1,i_2,i_3)\in\Hall(n,3)$, hence $[\sv_{i_1},\sv_{(i_2,i_3)}]=[\sv_{i_1},[\sv_{i_2},\sv_{i_3}]]=\sv_{(i_1,i_2,i_3)}$, from which it follows that $\Lambda_{\i\j}(\T)=\Lambda_{(i_1),(i_2,i_3)}(\T)=T_{(i_1,i_2,i_3)}$. 
 
 For $c>3$, set
 $ \eta_k:= [\sv_{i_k}, [ \sv_{i_{k+1}}, \dots, [\sv_{i_{c-1}}, \sv_{i_c}] \dots ]$.  It follows from~\eqref{Jac} that  
\[ [\sv_{i_1},[\sv_{i_{2}}, \dots, [\sv_{i_{c-1}}, \sv_{i_k}] \dots ]= [\sv_{i_1}, [\sv_{i_2}, \eta_3  ] ] 
= [ \sv_{i_2},  [\sv_{i_1}, \eta_3] ].  
\]
By repeating this process, we can keep swapping $\sv_{i_1}$ with subsequent terms $\sv_{i_3}, \dots$ until 
we arrive at $\sv_{i_r}$ where the process is terminated. Note that  
\[
(i_2, \dots, i_{r-1}, i_1, i_r, \dots, i_{c-1} )= \overline{(i_1, \dots, i_c)}.
\] 
This proves the lemma for the case $i_1
\ge i_{c-1}$.
 
 Let us now consider the case
$i_1<i_{c-1}$. In this case, the process described above can be continued all the way until $\sv_{i_1}$ arrives in the innermost commutator, that is, 
\[ [\sv_{i_1},[\sv_{i_{2}}, \dots, [\sv_{i_{c-1}}, \sv_{i_k}] \dots ]=
[\sv_{i_2}, [ \sv_{ i_3}, \dots, [\sv_{i_{c-2}}, [\sv_{i_1}, [\sv_{i_{c-1}}, \sv_{i_c}]  ] \dots   ]. \]
We can now use the  Jacobi identity to write
$ [\sv_{i_1}, [\sv_{i_{c-1}}, \sv_{i_c}] ]= [\sv_{ i_{c-1} }, [\sv_{i_1},\sv_{i_c}] ] - [\sv_{i_c}, [\sv_{i_{1}}, \sv_{i_{c-1}}] ]$.  
Similar to the previous case, by a repeated application 
of~\eqref{Jac} we can move $\sv_{i_c}$ inside
to the location so that the resulting sequence is decreasing. This proves the claim. 
\end{proof}
From now on we set 
$\widetilde{F}_{\g}(\T)$ to be equal to the submatrix of ${F}_{\g}(\T)$
that lies in the intersection of rows $\i\in\Hall(n,1)$ and columns $\j\in\Hall(n,c-1)$. 
We remark that
the only variables that appear in the entries of $\widetilde{F}_\g(\T)$ are the $T_\i$ for $\i\in\Hall(n,c)$. 

\begin{lemma}\label{rank1}
Let $K$ be any field and fix scalars $\alpha_i,\lambda_j\in K$ for $1\le i\le n$ and $0\le j\leq n-1$.
For each $ \i:=(i_1,\ldots,i_c) \in\Hall(n,c)$ set 
\begin{equation}
\label{eq:T'} T_{ \i}:= 
\displaystyle\left(
\prod_{k=1}^{c-2}  \alpha_{i_k}
\right) 
\begin{vmatrix}
\alpha_{i_{c-1} }  &  \alpha_{i_c}   \\
\lambda_{i_{c-1}-1} & \lambda_{ i_c-1}    \\
\end{vmatrix}.
\end{equation}
Set $\T:= (T_\i)_{\i\in\Hall(n,c)}$. 
Then the matrix $\widetilde{ F}_\g(\T)$ has rank at most $1$.  
\end{lemma}
\begin{proof}  
First we show that for every $ \i:=(i_1) \in \Hall(n,1)$ and $ \j:=(i_2, \dots, i_c) \in \Hall(n, c-1)$, 
\begin{equation}\label{main-formula}
\widetilde{F}_\g(\T)_{\i\j}^{}= 
\left(
  \prod_{k=1}^{c-2}  \alpha_{i_k} \right) \begin{vmatrix}
\alpha_{i_{c-1} }  &  \alpha_{i_c}   \\
\lambda_{i_{c-1}-1} & \lambda_{ i_c-1}    \\
\end{vmatrix}.
\end{equation}
We consider two different cases. If $i_1 \ge i_{c-1}$, then it follows from Lemma \ref{comcal} that 
 \[ { \Lambda}_{\i \j}( \T)= T_{(\overline{i_1, \dots, i_{c-1}}, i_c)}. \]
Since the sequence $(i_2, \dots, i_{c-1})$ is decreasing, and $i_1 \ge i_{c-1}$, it follows that $i_{c-1}$ 
is also the least element of $(\overline{i_1, \dots, i_{c-1}}, i_c)$ and hence
the first $c-2$ terms of $(\overline{i_1, \dots, i_{c-1}}, i_c)$ are precisely $i_1, \dots, i_{c-2}$, perhaps in a different order. 
Thus~\eqref{main-formula} follows from the fact that  the last two terms of the sequence are $i_{c-1}$ and $i_c$, in the same order.

If $i_1<i_{c-1}$, then it follows from Lemma~\ref{comcal} that 
\[{ \Lambda}_{\i \j}( \T)= T_{(i_2, \dots, i_{c-1}, i_1, i_c)}- T_{( \overline{i_2, \dots, i_{c-2}, i_c}, i_1, i_{c-1}) }. \]
Since both $(i_2, \dots, i_{c-1}, i_1, i_c)$ and $( \overline{i_2, \dots, i_{c-2}, i_c}, i_1, i_{c-1})$
are Hall sequences, we have 
\begin{equation}      
\begin{split}
{ \Lambda}_{\i \j}( \T) & = \alpha_{i_{c-1} } \cdot \left( \prod_{k=2}^{c-2 }  \alpha_{i_k} \right) \cdot \begin{vmatrix}
\alpha_{i_1}  &  \alpha_{i_c}   \\
\lambda_{ i_1-1} & \lambda_{i_{c}-1 }    \\
\end{vmatrix} 
 - \alpha_{i_c} \cdot  \left( \prod_{k=2}^{c-2 } \alpha_{i_k} \right) \cdot \begin{vmatrix}
\alpha_{i_1}  &  \alpha_{i_{c-1}}   \\
\lambda_{ i_1-1} & \lambda_{i_{c-1}-1 }    \\
\end{vmatrix} \\
&=  \left( \prod_{k=2}^{c-2 } \alpha_{i_k} \right) \cdot \left( \alpha_{i_{c-1}} \cdot \begin{vmatrix}
\alpha_{i_1}  &  \alpha_{i_c}   \\
\lambda_{ i_1-1} & \lambda_{i_{c}-1 }    \\
\end{vmatrix}  - \alpha_{i_c} \cdot \begin{vmatrix}
\alpha_{i_1}  &  \alpha_{i_{c-1}}   \\
\lambda_{ i_1-1} & \lambda_{i_{c-1}-1 }    \\
\end{vmatrix} \right)
\end{split}
\end{equation}
After applying the elementary identity
\[ r' \begin{vmatrix}
q  &  r   \\
s & t    \\
\end{vmatrix} - r \begin{vmatrix}
q  &  r'   \\
s & t'    \\
\end{vmatrix} = q \begin{vmatrix}
r'  &  r   \\
t' & t    \\
\end{vmatrix} \]
to the expression on the right hand side and merging the prefactor $q:= \alpha_{i_1}$ with the product 
$\prod_{k=2}^{c-2 } \alpha_{i_k} $ we obtain~\eqref{main-formula}. 
From~\eqref{main-formula} it follows that for $\i:=(i_1)\in\Hall(n,1)$ and $\j:=(i_2,\ldots,i_c)\in\Hall(n,c-1)$, we have $\widetilde{F}_\g(\T)_{\i\j}=\alpha_{i_1}\beta_\j$ for some $\beta_\j\in K$. Thus $\widetilde{F}_\g(\T)$ is
expressible as the product of the column vector $(\alpha_i)_{i\in\Hall(n,1)}$ and the row vector $(\beta_\j)_{\j\in\Hall(n,c-1)}$, so that  $\mathrm{rk}( \widetilde{F}_\g (\T))\leq 1$. 
\end{proof}


\begin{example}
We compute $\widetilde{F}_\g(\T)$ for $\g:=\m_{3,3}$ using Lemma \ref{comcal}. The rows are indexed by $(1),(2),(3)\in\Hall(3,1)$ and the columns are indexed by $(1,2),(1,3),(2,3)\in \Hall(3,2)$. Thus
\[
\widetilde F_\g(\T)=
\begin{pmatrix}
T_{(112)} & T_{(113)} & T_{(213)}-T_{(312)}\\
T_{(212)} & T_{(213)} & T_{(223)}\\
T_{(312)} & T_{(313)} & T_{(323)}\\
\end{pmatrix}.
\]
\end{example}

\begin{lemma} \label{combi}
Suppose $r, m,\delta \ge 1$. Let  $K$ be any  field.
\begin{enumerate}
\item[(a)] Suppose $f \in K[x_1, \dots, x_m]$ is a polynomial in $m$ variables over $K$ such  that the
degree of $f$ as a polynomial in $x_i$ is at most $\delta_i$ for all $1 \le i \le m$. Assume that 
$\#K\ge 1+ \max(\delta_1, \dots, \delta_m)$, and  $f(x_1, \dots, x_m)=0$ for all $(x_1, \dots, 
x_m) \in K^m$. Then $f$ is the zero polynomial. 
\item[(b)]  Suppose $f_1, \dots, f_r \in K[x_1, \dots, x_m]$, and the degree of each one of $f_1, \dots, f_r$ in the variable $x_i$ is at most $ \delta_i$ for $1\le i\le m$. 
Let $f: K^m \to K^r$ be the polynomial map defined by $f=(f_1, \dots, f_r)$. Assume that the $K$-subspace
of $K[x_1, \dots, x_m]$ spanned by $f_1, \dots, f_r$ has dimension at least $\delta$.  If $\#K\ge 1+ \max( \delta_1, \dots, \delta_m)$, then there exist
 $u_1, \dots, u_\delta \in K^m$ for which $f( u_1), \dots, f(u_\delta)$ are linearly independent vectors
in $K^r$. 
\end{enumerate}  
\end{lemma}

\begin{proof} Part (a) follows by induction on $m$. For $m=1$, the statement is clear. Assuming that (a) holds for $m-1$, write $f(x_1, \dots, x_m):= \sum_{ i =0}^{\delta_m } c_i x_m^i$ where the $c_i\in K[x_1,\ldots,x_{m-1}]$. Fix $(x_1,\dots, x_{m-1})\in K^{m-1}$ and consider the one-variable polynomial \[g(x_m):= f(x_1 ,\dots, x_{m-1}, x_m).
\] Since $g(x_m)=0$ for all $x_m \in K$, and 
$\#K > \deg g$, it follows that all the coefficients of $g$ vanish. The claim follows by the induction hypothesis applied to the $c_i$. 

For (b) 
by choosing a maximal linearly independent subset of $f_1, \dots, f_r$, we can assume that $\delta=r$. Thus the goal is to find  $u_1, \dots, u_r \in K^m$ for which $f( u_1), \dots, f(u_r)$ are linearly independent vectors in $K^r$. If this is not the case, then the image
of $f$ must lie in a proper subspace of $K^r$. This implies that there exists a non-zero vector 
$(c_1, \dots, c_r) \in K^r$ such that $c_1 f_1(u)+ \cdots + c_r f_r(u)=0$ for all $ u \in K^m$. By part (a), 
this implies that the polynomial $c_1 f_1+ \cdots + c_r f_r$ is the zero polynomial, which is a contradiction.
\end{proof}

We are now ready to prove Theorem~\ref{meta} in the case where $R$ is a finite field. In this case, we need to show that for $p$ sufficiently large we have
\begin{equation}
\label{eq:meta1bard}
\mf(\GG_{\F_q})= (c-1) {n+c-2 \choose c}  fq\qquad\text{ for all }f\ge 1,
\end{equation}
where $q:=p^f$.
Recall that $T_i$ is chosen as in \eqref{eq:T'}.
We set $\lambda_0:=0$ and $\lambda_1:=1$. Then the entries of $\widetilde{F}_\g$ are polynomials  in variables $ \alpha_1, \dots, \alpha_n$
and $ \lambda_2, \dots, \lambda_{n-1}$. Thus $\widetilde{ F}_\g$ takes values
in the vector space of matrices of size $n \times \# \Hall(n, c-1)$ with entries in the ring $\F_q[\alpha_1,\ldots,\alpha_n,\lambda_2,\ldots,\lambda_{n-1}]$.
First we prove the following claim:
\[
\text{The set $\{ T_\i\}_{\i \in \Hall(n,c) } $ is linearly independent in 
$\F_q[\alpha_1,\ldots,\alpha_n,\lambda_2,\ldots,\lambda_{n-1}]$. }
\]
To prove this claim suppose that 
\[ \sum_{ \i \in \Hall(n,c)} c_\i \cdot  \left(
\prod_{k=1}^{c-2}  \alpha_{i_k}\right)
 \cdot \begin{vmatrix}
\alpha_{i_{c-1} }  &  \alpha_{i_c}   \\
\lambda_{i_{c-1}-1} & \lambda_{ i_c-1}    \\
\end{vmatrix}=0, \]
for some coefficients  $c_\i \in \F_q$. We will show that all the coefficients must be zero. 
For each such linear dependence, define  $r$ to be the largest value of $i_c$ for which there exists $\i= (i_1, \dots, i_c) \in \Hall(n,c)$
such that $c_\i\neq 0$. Note that clearly $r \ge 2$. First suppose that $ r=2$. This implies that 
$ i_{c-1}=1$ and hence the linear dependence equation simplifies to 
\[ \sum_{ (i_1, \dots, i_{c-2}) \in \D(n, c-2)} c_{(i_1, \dots, i_{c-2},1,2) }  \prod_{k=1}^{c-1}  \alpha_{i_k}
=0. \]
Consider a monomial $g:=\alpha_{i_1}\cdots \alpha_{i_{c-1}}$ that appears in the above linear combination. Then $\deg_{\alpha_i}(g)=\#\{j:1\leq j\leq c-1\text{ and }i_j=i\}$. 
Since the sequence $(i_1,\ldots,i_{c-1})$ is decreasing, the sequence $(\deg_{ \alpha_1}(g), \dots, \deg_{ \alpha_n}(g))$
uniquely determines the values $i_1,\ldots,i_{c-2}$, hence uniquely determines $g$ as well. Thus the monomials in the above linear combination have distinct degree sequences and therefore 
there is no possibility of cancellation between them. 
Next suppose that $ r \ge 3$. This implies that there are some monomials involving $ \lambda_{r-1}$ and 
that there are no terms involving $ \lambda_j$ for $j >r-1$. It is also clear that the only terms involving 
$ \lambda_{r-1}$ correspond to those sequences $(i_1, \dots, i_c)$ for which $ i_c=r$ and hence
these terms are of the form
\[ \prod_{k=1}^{c-2}  \alpha_{i_k} ( \alpha_{i_{c-1}} \lambda_{r-1} - \alpha_r \lambda_{i_{c-1}-1}) . \]
By considering only the terms involving $ \lambda_{r-1}$ we obtain 
\[ \sum_{ (i_1, \dots, i_{c-1}) \in \D(n-1,c-1), i_{c-1}<r } c_\i  \prod_{k=1}^{c-1}  \alpha_{i_k} =0. \]
As in the previous case the monomials corresponding to different $\i$ are distinct, and hence this can 
only happen if $c_\i=0$ for all such terms. This contradicts the choice of $r$. 
This completes the proof of linear independence of the $\{T_\i\}_{\i\in\Hall(n,c)}$.

Next set
$f:=\widetilde{F}_\g(\T')$ in Lemma \ref{combi}(b), where $\T':=(T_\i)_{\mathbf i\in\Hall(n,c)}$. Note that the components of $f$ are the entries of 
$\widetilde{F}_\g(\T')$ which are
polynomials
in the $\alpha_i$ and the $\lambda_i$. 
Furthermore, for every $\i:=(i_1,\ldots,i_c)\in\Hall(n,c)$ if we set $\i':=(i_1)$ and $\j':=(i_2,\ldots,i_c)$ then by Lemma \ref{comcal}
we have $\widetilde{F}_\g(\T)_{\i'\j'}=T_\i$.
From  linear independence of the set $\{ T_\i\}_{\i \in \Hall(n,c) } $ and 
Lemma \ref{combi}(b)
it follows that  for $q\geq c$ one can find $ \# \Hall(n,c)$ linearly independent vectors $
\T':=(T_\i)_{\mathbf i\in\Hall(n,c)}$
with entries in $\F_q$
for which the values $\widetilde F_\g(\T')$ are linearly independent. By Lemma
\ref{rank1} it follows that  
 $\rank(\widetilde{F}_\g(\T'))=1$ for all such $\T'$. 
 We can extend each such vector $\T'$ to a vector $\T:=(T_\i)_{\i\in\Hall^{\geq 2}}$ by setting  
$T_\i=0$ for $\i\in\Hall(n,k)$, where $2\leq k\leq c-1$. For the latter vectors $\T$ we have 
$\rank_{\F_q} (F_\g(\T))=2$. Since $F_\g$ is skew symmetric and each variable $T_\i$
where $\i:=(i_1,\ldots,i_k)$ appears in the $(\i',\j)$-entry of $F_\g(\T)$, where $\i':=(i_1)$ and $\j:=(i_2,\ldots,i_k)$, we have $\rank_{\F_q}(F_\g(\T))\geq 2$ whenever $\T\neq 0$.  
 In view of Theorem \ref{Algorithm} this proves the assertion
in the case that $R$ is a finite field.

Let us now consider the case where $R=  \mathcal O/\mathfrak p^d$.  
Using Theorem \ref{thm:upperbound} and equality~\eqref{eq:meta1bard}
which was  just proved,  the upper bound follows immediately. Let us now turn to establishing the lower bound. 
We start with the following lemma.
Recall the notion of level of an element of $\mathcal O/\mathfrak p^d$ from Definition~\ref{dfn:level}.
\begin{lemma}\label{skew}
Let $B$ be an $r \times r$ skew-symmetric matrix with entries in 
$\mathcal O/\mathfrak p^d$. Suppose $B$ has an entry of level $\ell$. Then 
\[ \# \ker B \le p^{ f d(r-2) + 2 f\ell}. \]
\end{lemma}


\begin{proof}
Write $B=(B_{ij})_{ 1 \le i, j \le r}$. If $\ell=d$, there is nothing to prove. Next assume that the entry in question is non-zero, and hence non-diagonal. After possibly permuting rows and columns, we can assume that $\lev(B_{12})= \ell$. Consider a vector $(x_1,\ldots,x_r) \in (\mathcal O/\mathfrak p^d)^{\oplus r}$ in $\ker B$. Then the following equations hold:
 \begin{equation*}
B_{12} x_2  = - \sum_{ j =3}^{r} B_{1j} x_j\quad \text{and}\quad 
B_{21} x_1  = - \sum_{ j =3}^{r} B_{2j} x_j.
\end{equation*}
Once the values of the $x_j$ for $ 3 \le j \le r$ are set, the number of choices for 
each one of $x_1$ and $x_2$ is $p^{f\ell}$. Since the number of choices for
$(x_3, \dots, x_r)$ is $p^{fd (r-2)}$, the claim follows immediately. 
\end{proof}

In order to avoid confusion with parameters $n$ and $ c$ of $\m_{n,c}$, henceforth we will denote the parameters  $l_1, l_2, n, m$ 
that were associated to $\g$ in Section~
\ref{Sec:polyfinfil}
by 
$\underline l_1, \underline l_2, \underline n, \underline m$. 
Thus in particular we have 
\[
\underline m=\#\Hall^{\geq 2}\quad,\quad 
\underline l_1=\#\Hall(n,c)=(c-1) {n+c-2 \choose c}\quad\text{and}\quad\underline l_2=0.
\] Moreover, the commutator matrix $ F_\g( \T)$ is an $\underline n \times \underline n$ matrix with $\underline n = \rank_\Z (\m_{n,c})- \underline l_1$.

\begin{lemma}\label{kernelskew}
Let  $ \mathbf{a}:=  
( \mathbf{a}(\i))_{ \i \in
\Hall^{\geq 2}}\in   (\mathcal O/\mathfrak p^d)^{\oplus \underline m}$. Then 
\[ \#\ker_{\mathcal O/\mathfrak p^d}F_\g(  \mathbf{a}) \le \min_{ \i \in \Hall(n,c) } p^{fd(\underline n-2) +2f \lev ( \mathbf{a} ( \i) ) }. \]
\end{lemma}

\begin{proof}
By Lemma~\ref{bases} we have $\underline n=\sum_{j=1}^{c-1} \#\Hall(n,j)$ and the rows and columns of $F_\g(\mathbf a)$ are indexed by elements of $\Hall^{\le c-1}=\bigcup_{j=1}^{c-1}\Hall(n,j)$. Now fix any $\i',\j'\in\Hall^{\leq c-1}$. 
By Lemma \ref{skew}  for $B:=F_\g(\mathbf{a})$, $r:=\underline n$, and $\ell:=\lev(F_\g(\mathbf a)_{\i' \j'})$ we have
\[
 \#\ker_{\mathcal O/\mathfrak p^d}F_\g(  \mathbf{a}) \le  p^{fd(\underline n-2) +2f \lev ( F_\g(\mathbf a )_{\i'\j'}) }.
\]
In particular, if for any 
$\i\in\Hall(n,c)$ we set 
$\i':=(i_1)$ and 
$\j':= ( i_2, \dots, i_c)$
then by Lemma~\ref{comcal}
we obtain 
$\#\ker_{\mathcal O/\mathfrak p^d}F_\g(  \mathbf{a}) \le  p^{fd(\underline n-2) +2f \lev ( \mathbf{a} ( \i) ) }$. Finally, by taking the minimum over all $\i\in\Hall(n,c)$ we obtain the assertion of the lemma. 
%
\end{proof}

Let $S=\{ \mathbf a_1,\ldots,\mathbf a_{fe \underline l_1} \} \subseteq (\mathcal O/\mathfrak p^d)^{\oplus\underline m}$ be a set of vectors that corresponds to $\mf(\GG_R)$  in 
Proposition~\ref{Algorithm-R} (thus we assume that $p$ is sufficiently large accordingly). 
Let $ \pi_\i$ for $\i\in\Hall(n,c)$ denote the natural  projection 
\[
\bigoplus_{\i\in\Hall(n,c)}(\mathcal O/\mathfrak p^e)\to \mathcal O/\mathfrak p^e,
\] 
that maps $\mathbf a$ to $\mathbf a(\i)$.  
Using 
Lemma \ref{Lablace-identity},
we partition $S$ into sets $S_\i$, each of cardinality $ef$, such that $ \pi_\i ( {\mathsf{proj}}( S_\i) )$ is a basis
for the $\F_p$-vector space $ \mathcal O/\mathfrak p^e$. By Lemma 
\ref{kernelskew} we have
\begin{align}\label{firsteq}
\mf(\GG_R)&=\notag\sum_{i=1}^{fe \underline l_1}
\left(
\frac{p^{fd\underline n}}{\#\ker_{\mathcal O/\mathfrak p^d}(F_\g(\mathbf a_i))}
\right)^\frac12\\
&=\sum_{ \i \in \Hall(n, c) } \sum_{ \mathbf{a} \in S_\i}  \left( \frac{p^{fd\underline n}}{\#\ker_{\mathcal O/\mathfrak p^d}(F_\g(\mathbf a ))} \right)^{\frac12}
\ge \sum_{ \i \in \Hall(n, c) } \sum_{ \mathbf{a} \in S_\i} p^{f( d- \lev( \mathbf{a}(\i) ) } ).
\end{align}
Recall that for every $ 0 \le \ell \le e-1$, the set of $x \in  (\mathcal O/\mathfrak p^e)$ with $\lev(x) \ge \ell$ is an $\F_p$- subspace of dimension $ f (e-\ell)$.  
Let us call this subspace $\mathcal W_\ell$. 
 Fix $ \i \in \Hall( n,c)$, and set $a_\ell= \frac{1}{f} \# \{ \mathbf{a}  \in S_{\i}: \lev( \mathbf{a}(\i)) = \ell \}$. 
Since $ \{ \pi_\i(\mathsf{proj}(\mathbf{a})): \mathbf{a} \in S_\i \}$ is a basis for $ \mathcal O/\mathfrak p^e$, it follows that for every $ 0 \le \ell \le e-1$, the 
set of $\mathbf{a} \in S_\i$ with 
$ \lev(\mathbf{a}(\i)) \ge \ell$ is a linearly independent subset of $\mathcal W_\ell$. Thus, the cardinality of the latter set of vectors is at most $\dim \mathcal W_\ell=f( e- \ell)$. Hence
$a_\ell+ \cdots + a_{e-1} \le e- \ell$ for all $ 0 \le \ell \le e-1$. 
Using Lemma~\ref{ineq} for $m=e$ and $x_\ell= p^{f(d- \ell)}$, where $ 0 \le \ell \le e-1$, we deduce
\begin{equation}\label{secondeq}
 \frac{1}{f} \sum_{ \mathbf{a} \in S_\i} p^{f( d- \lev( \mathbf{a}(\i) )) }  = \sum_{ \ell =0}^{e-1} a_\ell x_\ell   \ge 
\sum_{\ell=0}^{e-1} x_\ell= \sum_{\ell=0}^{e-1} p^{f(d-\ell) }. 
\end{equation} 
Equivalently, $ \sum_{ \mathbf{a} \in S_\i} p^{f( d- \lev( \mathbf{a}(\i) ) )}   \ge f \sum_{\ell=0}^{e-1} p^{f(d-\ell) }. $
The assertion of Theorem~\ref{meta} follows immediately from combining \eqref{firsteq} and \eqref{secondeq}.

\bibliographystyle{amsalpha}
\bibliography{Ref.bib}

\end{document}